\documentclass[11pt]{article}
\usepackage{format}
\usepackage[english]{babel}
\usepackage[utf8]{inputenc}

\usepackage[sort, numbers]{natbib}
 
\usepackage{preamble}

\title{\vspace{-1in}\huge
A generalized moment approach to sharp bounds for conditional expectations}

\author{
\normalsize Wouter J.E.C. van Eekelen \vspace{-1mm}\\ 
\scriptsize Booth School of Business, University of Chicago, \href{Wouter.vanEekelen@chicagobooth.edu}{Wouter.vanEekelen@chicagobooth.edu}   \vspace{.5mm} \\ 
} 

\date{\small\mydate{\today}}

\linedabstractkw{
In this paper, we address the problem of bounding conditional expectations when moment information of the underlying distribution and the random event conditioned upon are given. To this end, we propose an adapted version of the generalized moment problem which deals with this conditional information through a simple transformation. By exploiting conic duality, we obtain sharp bounds that can be used for distribution-free decision-making under uncertainty. Additionally, we derive computationally tractable mathematical programs for distributionally robust optimization (DRO) with side information by leveraging core ideas from ambiguity-averse uncertainty quantification and robust optimization, establishing a moment-based DRO framework for prescriptive stochastic programming. 
}{Conditional expectations, sharp bounds, conic duality,  prescriptive stochastic programming, distributionally robust optimization with side information} 

\begin{document}

\maketitle

\section{Introduction}
Distribution-free performance analysis of stochastic models strives to obtain tight bounds for the expectation of an objective function of random variables, using only limited information about the underlying probability distributions. Traditionally, given the moment information about the random variables, these problems are modeled as generalized moment problems. The sharpest (i.e., “best possible”) upper and/or lower bounds for these objective functions of random variables are found by solving semi-infinite optimization problems, where the optimization is taken over all admissible distributions of the random variables. In this paper, we explore a new situation in which, besides the given moment sequence, we possess stronger information---that is, we have knowledge that a particular random event will occur, which pertains to the realizations of the random variables rather than their underlying probability distribution. We are interested in solving the moment problem with the knowledge of this random event, which may be based on assumptions, assertions, expert judgment or past observations. Alternatively, we may want to determine the worst-case behavior of a stochastic model if such a random event occurs. Conditional expectations offer a way to model this event information and provide us with the best estimate of the expected value of a function of random variables, knowing that the specified random event will occur.
It is of interest to define a generalized moment framework for this new setting with conditional expectations, instead of standard expectations in which the random variables are not restricted to take on a subset of the values in the event set. However, random variables conditioned on a random event give rise to a different probabilistic concept, requiring a distinct type of analysis compared to the traditional theory on generalized moment problems.

Moment problems have been investigated by probability theorists since the end of the 19th century, see, e.g., \cite{chebyshev1874valeurs,stieltjes1894recherches,stieltjes1895recherches,markov1884certain}. In its most classical form, the problem of moments is a feasibility problem that aims to determine if there exists a probability distribution that satisfies the given distributional information. Chebyshev formulated the problem of finding the optimal bound, given only mean-variance information, for the tail distribution of a random variable, which his student Markov later solved using continued fractions techniques, ultimately resulting in the well-known Chebyshev inequality.
It took till the 1950s and 60s for the interest in this type of literature to be rekindled, which gave rise to the vast literature on Chebyshev systems. We refer to the monograph of \citet{karlin1966tchebycheff} for a comprehensive review.
As a prominent reason for the renewed interest in the generalized moment problem, we highlight the advent of duality theory as a novel method for solving this problem, as demonstrated in works such as \cite{isii1962sharpness,karlin1966tchebycheff,kemperman1968general}. As summarized in \cite{karlin1966tchebycheff}, \citet{marshall1960multivariate} were the first to generalize Chebyshev's inequality to the multivariate setting. Since the traditional Chebyshev inequalities only using moment information are considered too loose to be used for practical purposes, \citet{mallows1956generalizations} used structural properties to sharpen these bounds. In more recent research, the connections between moment problems, nonnegativity of polynomials, and semidefinite programming have been exploited (see, e.g., \cite{parrilo2000structured,lasserre2001global,parrilo2003semidefinite,bertsimas2005optimal,popescu2005semidefinite}). \citet{shapiro2001duality} discussed the relationship between duality results for generalized moment problems and the more general theory of conic duality, extending these concepts to more general topological vector spaces.  \citet{smith1995generalized} revisited the generalized moment problems in a contemporary discussion, highlighting its various applications in decision theory. In his work, Smith briefly mentions the setting with prior information, but notes that the resulting expectation is no longer linear in the probability measure, thus presenting a more challenging problem that is not directly amenable to the techniques discussed in his work. As a solution, \citet{smith1995generalized} suggested a linearization technique, as discussed in \cite{lavine1991sensitivity} and \cite{cozman1999calculation}, which converts the problem into a series of regular generalized moment problems.
Along the lines of \citet{shapiro2001duality}, we define a novel variant of the generalized problem of moments and use semi-infinite programming theory to obtain tight bounds in a setting where the objective function is a fractional, rather than linear, function of the probability distribution. A comprehensive overview of general semi-infinite programming theory can be found in \cite{hettich1993semi}. To the best of our knowledge, there are no general techniques available for obtaining tight bounds on conditional expectations. Although \citet{mallows1969inequalities} provided some bounds for conditional expectations for traditional power moments, these are only for the expectation of a single random variable and are not necessarily tight. Moreover, \citet{mallows1969inequalities} primarily used conditional information to sharpen Chebyshev-type tail probability bounds. In contrast, we provide a general framework for obtaining the best possible bounds for conditional expectations under limited information.

Extending the ideas described above to the decision-making setting, we are interested in obtaining tight bounds for the expectation of some objective function of random variables, where the function also contains a decision vector, in order to protect ourselves against the worst nature has to offer.
\citet{Scarf1958} was arguably the first to bring the concept of minimax optimization into the area of operations research. Scarf considered a single-item newsvendor problem where the demand distribution is not precisely known, but only characterized by limited information. The set of admissible distributions, also known as the {ambiguity set}, contained all distributions with specified mean and variance. Scarf used duality techniques to find the worst-case distribution that maximizes the total costs of the newsvendor, and then determined the order quantity such that these maximal total costs were minimized. Subsequently, generalized moment problems have been used to solve such minimax optimization problems in a vast amount of literature (see, for instance, \cite{vzavckova1966minimax,birge1987computing,birge1991bounding,dupavcova1987minimax,shapiro2002minimax,shapiro2006worst}).
Minimax optimization evolved into the field of distributionally robust stochastic optimization, or distributionally robust optimization (DRO). This term was first coined by \citet{delage2010distributionally}, and has since then become the standard terminology in the operations research community.
Distributionally robust optimization is advocated as the unifying paradigm between two distinct fields for decision-making under uncertainty: robust optimization and stochastic programming.
Robust optimization \citep{ben1998robust,ben2009robust} provides an effective way to deal with problems subject to parameter uncertainty through uncertainty sets. Although encoding uncertainty through these uncertainty sets often results in computationally tractable problems, the resulting solutions might turn out to yield overly conservative outcomes. In contrast, stochastic programming captures parameter uncertainty with full distributional information, but often yields intractable models \citep{shapiro2009lectures}. Ambiguity sets provide a powerful modeling tool for capturing distributional information about uncertain parameters in terms of their support and descriptive statistics. A significant portion of the literature on distributionally robust optimization (DRO) is moment-based, with partial information given by means, moments, and dispersion measures \citep{popescu2007robust,delage2010distributionally, wiesemann2014distributionally, hanasusanto2017ambiguous}. As adequately described in \cite{wiesemann2014distributionally}, the tractability of moment-based DRO problems relies on the intricate interplay between the objective function and the structure of the ambiguity set. Other popular ambiguity sets are based on statistical-distance measures, which restrict their members to be within a specific distance of a reference distribution. Examples of these statistical-distance measures include $\phi$-divergences and the Wasserstein distance (see, e.g., \cite{ben2013robust,mohajerin2018data}). In the present work, we focus on moment-based information.

Conditional-moment information has been used in various works to describe ambiguity sets. In the minimax stochastic programming literature, \citet{birge1987computing} incorporated such information into the constraints of generalized moment problems to bound the objective values of stochastic programming problems. \citet{de2020distributionally} restricted the distributions in the ambiguity set to polynomial density functions. They demonstrated that these ambiguity sets are highly expressive because they can conveniently accommodate distributional information about conditional probabilities, conditional moments and marginal distributions. \citet{chen2020robust} introduced scenario-wise ambiguity sets that capture information with conditional expectation constraints based on generalized moments. Although these works use conditional moment information, their objectives aim to maximize the conventional expectation. Consequently, their approaches cannot directly handle the case in which the conditional expectation is the objective.

Here we should mention some work on distributionally robust optimization in which the objective function is fractional. When only support information is available, \citet{gorissen2015robust} extended robust optimization formulations to fractional programming in which the objective function is a fraction of two functions of the uncertain parameters. \citet{liu2017distributionally} solved a DRO problem with moment constraints which consists of the maximization of ambiguous fractional functions representing reward-risk ratios. \citet{ji2021data} also investigated this class of fractional DRO problems using semi-infinite programming epigraphic formulations to solve the ambiguous reward-risk ratio problem and, additionally, design a data-driven formulation and solution framework using the Wasserstein ambiguity set.
Using the conditional information, we can use the conditional-expectation bounds to enhance “worst-case” decision-making in a DRO framework. As for this setting, there does exist a separate thrust of research named contextual distributionally robust optimization, in which the prior knowledge on the realizations of the random variables is commonly referred to as {side information}. This area of research is a natural extension of the prescriptive stochastic programming paradigm. In this paradigm, the central object of interest is the joint distribution of the side information and the outcome random variables, which, if known, would result in more accurate estimations of the outcome variable when conditioning this distribution on the side information given. In practice, however, this joint distribution is usually not known precisely and is only estimated using a finite data sample. The ultimate goal of prescriptive stochastic programming is to develop an optimization methodology that uses the available side information to improve decision-making given only limited insights into the predictive power of the side information on the uncertain outcome parameters (see, e.g., \cite{ban2019big,bertsimas2020predictive,srivastava2021data}). 
The contextual DRO modeling paradigm assumes that, next to this side information, the joint distribution is contained in an ambiguity set that is defined using the limited information available. Research on this paradigm is still relatively scarce. We highlight a number of works. \citet{esteban2022distributionally} described ambiguity through a partial mass transportation problem, and exploited probability-trimming methods to solve the contextual DRO problem.
In contrast, \citet{nguyen2021robustifying} worked directly with the optimal transport ambiguity set, and the authors succeeded in finding tractable conic reformulations for DRO problems with side information.
\citet{nguyen2020distributionally} considered a Wasserstein-ball ambiguity set \citep{mohajerin2018data}, which is centered on the empirical distribution that follows from the available data sample of the side information and the outcome parameters. Even though the Wasserstein-ball ambiguity set is a class of distributional ambiguity sets obtained through the theory of optimal transport, the models and results obtained in \cite{nguyen2021robustifying,nguyen2020distributionally} behave qualitatively differently due to special properties of the type-$\infty$ Wasserstein distance, which is used to construct the ambiguity set.
The literature described above mostly considers ambiguity sets that are defined through measures that define distances between probability distributions (such as the Wasserstein metric), whereas in this paper, we shall focus on ambiguity sets described by generalized moment information.


The main contributions of the paper may be summarized as follows.
\begin{enumerate}
    \item[1.] We expand the theory on semi-infinite programming and generalized moment problems, by deriving duality results for linear-fractional programming in the semi-infinite setting. We extend several well-known results from the theory on generalized moment problems in order to bound conditional expectations, rather than standard expectations of functions of random variables. The fact that most results carry over to this setting with conditional expectations is particularly intriguing because the conditional expectation is a nonlinear function of the probability measure (thus not amenable to standard techniques based on semi-infinite linear programming).
    \item[2.] We apply these novel results for the generalized conditional-bound problem to univariate functions of a simple random variable. Using primal-dual arguments, we obtain several closed-form bounds for different types of dispersion information. In addition to generalized moment information, we show that structural properties of the distribution can also be included. We further demonstrate our approach by resolving a minimax optimization problem, taken from the robust monopoly pricing literature. It is further asserted that most computations, in the univariate setting, are as tractable as with the linear expectations operator. 
    \item[3.]  We use findings from robust uncertainty quantification and distributionally robust convex optimization to develop conic reformulations for the multivariate problem. We then apply these reformulations to contextual DRO, presenting a generalized moment framework for distributionally robust optimization with side information. The resulting framework is designed for conditional decision-making, incorporating both the side information and the distributional information contained within the ambiguity set. The computational tractability of the reformulations turns out to be closely related to that of distributionally robust convex optimization problems with support restrictions on the random variables.
\end{enumerate}

The remainder of the paper is organized as follows. In Section~\ref{section2}, we introduce the generalized moment bound problem for conditional expectations and elaborate on the duality approach. Section~\ref{section3} discusses several examples of tight bounds for the conditional expected value of functions of a single random variable. In Section~\ref{section4}, the moment-based contextual DRO framework is presented. Most proofs of minor results are deferred to the Appendix. Finally, in Section~\ref{section5}, we conclude and provide several directions for future research. 


\section{A duality framework for generalized conditional-bound problems}\label{section2}
We first describe the problem of bounding conditional expectations in Section~\ref{sec:generalframe} and subsequently derive the associated dual problem in Section~\ref{sec:dualprob}. Then, in Section~\ref{sec:mainresults}, we provide fundamental results that will be employed in later sections to obtain the desired sharp bounds.

\subsection{Problem statement}\label{sec:generalframe}
We aim to find the best upper bound for the conditional expectation of a random vector $X$. Let us first introduce some notation. Let $\E$ denote the expectation operator, and $g(\cdot)$ denote an arbitrary measurable function of $X$. The random vector $X$ is defined on the support $\Omega\subseteq\mathbb{R}^n$, which we assume is a closed set endowed with the Borel sigma algebra $\mathcal{B}_\Omega$. The random vector $X$  is governed by a probability measure $\P:\mathcal{B}_\Omega\to[0,1]$, such that for a measurable set $\mathcal{S}\in\mathcal{B}_\Omega$ we have $\P(\mathcal{S})=\P(X\in\mathcal{S})$. Furthermore, $\P$ lies in some convex set of probability measures $\cP$. Throughout the paper, the terms “probability measure” and “probability distribution” are used interchangeably. We assume that $\event\in\mathcal{B}_\Omega$ is an arbitrary measurable event modeling the random event observed, pertaining to the realization of $X$. Let $\event$ be a set with strictly positive measure $\P(X\in\event)>0$ so that $\E[g(X) \mid \event]$ is well-defined and denotes the conditional expectation of $X$ restricted to the values in the set $\event$. We now have the necessary notation to develop an adapted version of the generalized moment problem that incorporates random events or, using different terminology, side information. The central problem in this paper can be formulated as follows:
\begin{equation}\label{eq:generalmomentproblem}
\sup_{\P\in \mathcal{M}_+(\Omega)}  \E_\P[g(X) \mid X\in\event] \quad  \text{subject to} \quad  \E_\P[h_j(X)] = q_j \text{ for } j = 0,\dots,m,
\end{equation}
where $\mathcal{M}_+(\Omega)$ denotes all nonnegative measures defined on the support $\Omega$, and $g,h_0,\dots, h_m$ are real-valued, measurable functions that model the objective function and the available (generalized) moment information. The probability mass constraint is explicitly included as $h_0 \equiv 1$ and $q_0 = 1$. Let $\cP$ denote the ambiguity set that contains the true probability distribution. 
For a given moment vector $\vec{q}$, define the set
\begin{eqnarray}\label{eq:cone}
\cP(\vec{q}):=\big\{\mathbb{P} \in \cP_0(\Omega): \int h_j(x) \mathrm{d}\mathbb{P}(x)=q_j, j=0, \ldots, m\big\},
\end{eqnarray}
which contains all probability distributions that comply with the given moment sequence. Here, we typically assume $\P$ is an element of a set of probability measures $\cP_0(\Omega)$ with support contained in $\Omega$. Thus, the constraint $h_0 \equiv 1$ is implicitly assumed so as to normalize the measures in $\mathcal{M}_+(\Omega)$ to obtain probability distributions. As a closely related concept, we consider the cone of moments $\vec{q}\in\mathbb{R}^m$ that yield a nonempty ambiguity set $\cP(\vec{q})$, which can be defined as
$$
\mathcal{Q}:=\left\{\vec{q}\in\mathbb{R}^m : \exists\P\in\mathcal{M}_+(\Omega) \text{ such that }\P\in\ \cP(\vec{q})\right\}.
$$ 
This set thus contains all  moment vectors $\vec{q}$ for which \eqref{eq:generalmomentproblem} has a solution. That is, the moment constraints are consistent or, in other words, there exists a probability distribution feasible  to the generalized moment problem. We henceforth assume that the moment constraints are consistent so that the ambiguity set $\cP$ is nonempty, and therefore a feasible solution to problem \eqref{eq:generalmomentproblem} always exists.

Now that we have introduced the required notation and studied the constraints of \eqref{eq:generalmomentproblem}, let us turn to the objective function. Instead of the regular expectation of a function of a random vector studied in generalized moment problems, $\E[g(X)]$, we now study 

\begin{equation}\label{eq:objective}
\E_\P[g(X)  \mid X\in\event] = \int_\event g(x)\, \dF{x} = \frac{\E_\P[g(X)\1_{\event}(X)]}{\E_\P[\1_{\event}(X)]}, 
\end{equation}
in which $\mathbb{Q}$ denotes the conditional probability measure, given that it exists. Notice that the objective function is a fractional, and thus nonlinear, function of the probability distribution $\P$. As a result, \eqref{eq:generalmomentproblem} belongs to the class of distributionally robust fractional optimization problems, which are fundamentally more difficult to solve than the standard problem that simply maximizes the expectation of a function (see, e.g., \cite{liu2017distributionally,ji2021data}).

\subsection{An equivalent problem and its dual}\label{sec:dualprob}
Problem \eqref{eq:generalmomentproblem} can be formulated equivalently as a semi-infinite linear-fractional program (LFP). The semi-infinite LFP reformulation of \eqref{eq:generalmomentproblem} is given by

\begin{equation}\label{eq:condmomentprob}
\begin{aligned}
&\sup_{\P(x)\geq0} & & \frac{\int_\Omega g(x)\1_{\event}(x)\Dist{x}}{\int_\Omega \1_{\event}(x)\Dist{x}} \\
&\text { s.t. } &  &\int_\Omega h_j(x) \d\P(x) = q_j,\  \forall j=0,\ldots,m,\\
&&& \P(X\in\event)>0,
\end{aligned}
\end{equation}
in which $\1_{\event}(x)$ equals 1 if $x\in\event$, and 0 otherwise. Here the optimization of the linear-fractional objective is taken over infinite-dimensional variables (i.e., the probability distribution). We further have a finite number of constraints that describe the moment information. The final constraint ensures that the conditional expectation is well-defined by avoiding conditioning on a set of measure zero.
In the finite setting, these linear-fractional programs can be reduced to linear programs through a Charnes-Cooper transformation \cite[Theorem~2]{charnes1962programming}. If we generalize this to infinite-dimensional spaces, the Charnes-Cooper transformation becomes
\begin{equation}\label{eq:charnescooper}
\begin{aligned}
 \frac{\dF{x}}{\Dist{x}} = \alpha, \text{ with } \alpha = \frac{1}{\int_\Omega \1_{\event}(x)\Dist{x}}. 
\end{aligned}
\end{equation}
In some sense, this generalized Charnes-Cooper transformation constitutes a change of measure, from the original probability measure $\P$ to its conditional counterpart $\mathbb{Q}$. The variable $\alpha$ is a scaling parameter that essentially models the normalization on the random event.

After transformation \eqref{eq:charnescooper}, problem \eqref{eq:condmomentprob} reduces to the semi-infinite linear program (LP)
\begin{equation}\label{eq:generalmomentproblem2}
\begin{aligned}
&\sup_{\alpha,\mathbb{Q}(x)\in\mathbb{R}_+} &  &\int_\Omega g(x) \1_{\event}(x)\dF{x}\\
&\text{ s.t. } &      & \int_\Omega h_j(x) \, \dF{x}=\alpha q_j,\quad \forall j=0,\ldots,m, \\  &&&  \int_\Omega\1_{\event}(x) \dF{x}=1  . 
\end{aligned}
\end{equation}
where all the right-hand sides of the constraints in \eqref{eq:generalmomentproblem2} are scaled by $\alpha$, and the last line ensures that $\mathbb{Q}$ is a proper (conditional) distribution when defined on its support $\event$. To determine the dual of \eqref{eq:generalmomentproblem2}, one can employ semi-infinite linear programming duality, as in Section~6 of \cite{hettich1993semi}. It then follows from standard calculations that the Lagrangian dual of \eqref{eq:generalmomentproblem2} is
\begin{equation}\label{eq:generaldual}
\begin{aligned}
&\inf_{\lambda_0,\dots,\lambda_{m+1}} &    & \lambda_{m+1}\\
&\text{ s.t. } &      & \sum_{j=0}^m  \lambda_j q_j \leq 0, \\
& &      & \sum_{j=0}^m \lambda_j h_j(x) + \lambda_{m+1} \1_{\event}(x)  \geq g(x)\1_{\event}(x), \ \forall x\in\Omega,
\end{aligned}
\end{equation}
where the dual variables $\lambda_0,\dots,\lambda_{m+1}$ are associated to each constraint in the primal \eqref{eq:generalmomentproblem2}. From this point on, we use the shorthand notation $h_{m+1}(x):=\1_{\event}(x)$. Notice that the dual problem has a finite number of decision variables, but an infinite number of constraints. By virtue of weak duality, an upper bound for \eqref{eq:generalmomentproblem2} follows from a feasible solution to \eqref{eq:generaldual}. The optimal dual solution to problem \eqref{eq:generaldual} yields a viable upper bound by weak duality, but whether this bound is sharp (i.e., whether strong duality holds) is still an open question, to which we will seek an answer in the next subsection. 

\subsection{Strong duality of the generalized conditional-bound problem}\label{sec:mainresults}
The purpose of this section is twofold: we demonstrate that \eqref{eq:generaldual} is strongly dual to \eqref{eq:generalmomentproblem}, also extending this result to more general, convex sets of probability measures $\cP_0$ that might model structural properties such as symmetry and unimodality, and we show that \eqref{eq:generalmomentproblem} can be reduced to a finite-dimensional problem in which one optimizes over a parametric family of distributions.

Since the objective function of \eqref{eq:generalmomentproblem} is nonlinear with respect to the expectation operator, it is not immediately clear how to pass down sufficient conditions regarding strong duality for generalized moment problems to the setting with conditional expectations.
Therefore, in order to prove our main results, we use an alternative formulation of problem \eqref{eq:condmomentprob}. The equivalence of this formulation is provided by the following lemma, which strongly hinges on the relation between fractional and parametric nonlinear programming \citep{dinkelbach1967nonlinear,gorissen2014deriving}. 
\begin{lemma}\label{lemma:equiv}
Suppose that problem \eqref{eq:condmomentprob} has a finite optimal value.  Then problem \eqref{eq:condmomentprob} is equivalent to
\begin{equation}\label{eq:altern}
\begin{aligned}
&\inf_{\tau \in  \mathbb{R}}&  & \tau \\
&\textnormal{ s.t. } & & \sup_{\P \in \mathcal{P}} \mathbb{E}_\P\left[g(X)\1_{\event}(X) -\tau \1_{\event}(X)\right] \leq 0.
\end{aligned}
\end{equation}
That is, the optimal value $\tau^*$ agrees with that of \eqref{eq:generalmomentproblem}, and the suprema in both \eqref{eq:generalmomentproblem} and \eqref{eq:altern} are achieved exactly by the same extremal distribution $\P^*$, or achieved asymptotically by an identical sequence of distributions $\{\P_k^*\}_{k\geq1}$.
\end{lemma}

This result is closely related to Proposition~2.1 in \cite{liu2017distributionally}, which deals in a similar fashion with a class of fractional DRO problems.
It is evident that \eqref{eq:altern} is significantly easier to work with than the original semi-infinite formulation \eqref{eq:condmomentprob} since the problem is linear with respect to $\P$, rather than linear-fractional. This can be seen from the constraints of \eqref{eq:altern} in which the expectation $\E_\P[\cdot]$ appears linearly, whereas in \eqref{eq:objective}, we have a fraction of expectation operators.
From the duality theory of moment problems, solving problem \eqref{eq:altern} turns out to be equivalent to solving the dual \eqref{eq:generaldual}. As a consequence, solving \eqref{eq:condmomentprob} will also be equivalent to solving \eqref{eq:generaldual} due to the equivalence between \eqref{eq:condmomentprob} and \eqref{eq:altern}.
We next show strong duality holds. To this end, we make the following assumptions:
\begin{enumerate}
    \item[(A1)] The function $g(x)$ is bounded on the support $\Omega$.
    \item[(A2)] There exists a positive number $\epsilon>0$ such that
    $$
    \inf_{\P\in\cP}\E_{\P}[\1_{\event}(X)] \geq \epsilon.
    $$
    \item[(A3)] The Slater condition holds; that is, $\vec{q}\in\operatorname{int}(\mathcal{Q}_{\cP})$, where “$\operatorname{int}$” denotes the interior of a set.
\end{enumerate}
\vspace{0.25cm}
Assumption (A1) is standard and could be relaxed. It is satisfied, for example, when $\Omega$ is a compact set and $g(x)$ is continuous, by virtue of Weierstrass' Extreme Value Theorem. This assumption is used to guarantee that the optimal value of \eqref{eq:condmomentprob}  is finite.
Assumption (A2) provides a sufficient condition for the conditional expectation to be well-defined, as it avoids conditioning on a set of measure zero.
Assumption (A3) constitutes a Slater-type condition on the moments of $X$, which is standard for generalized moment problems; see, for example, Proposition~3.4 in \cite{shapiro2001duality}. 
Under these regularity conditions, we show strong duality holds for a general, convex set of probability distributions $\cP_0$, possibly endowed with structural properties (e.g., symmetry and unimodality). 
We will focus on structural ambiguity sets $\cP_0(\Omega)$ that possess a mixture representation. In other words, we assume $\cP_0$ can be “generated” by a convenient class of distributions, say $\mathcal{T}$, such that every distribution $\mathbb{P}\in\cP_0$ can be written as a mixture (i.e., an infinite convex combination) of the extremal distributions (i.e., the extreme points of the convex set $\cP_0$) that constitute $\mathcal{T}$. For every Borel set $\mathcal{S} \in \mathcal{B}_\Omega$, it should thus hold that
$$
\mathbb{P}(X\in\mathcal{S})=\int_{\mathcal{T}} \mathbb{T}_\vec{t}(X\in\mathcal{S}) \mathrm{d}\mathbb{M}(\vec{t}),
$$
where $\mathbb{T}_\vec{t} \in \mathcal{T} \subseteq \cP_0$, is a parameterized representation of the family of extremal distributions of $\mathcal{P}_0$, and $\mathbb{M}$ represents the mixture distribution that generates $\mathbb{P}$ from the extremal distributions in $\mathcal{T}$. This finite-dimensional parameterization of the family of extremal distribution will prove useful when determining the optimal bounds. For a thorough discussion on these structural ambiguity sets in the context of DRO, we refer to the work of \citet{popescu2005semidefinite}. We use these general ambiguity sets with structural properties when formulating our main results.

Finally, before formulating our main result, we introduce some final technical notation from conic duality theory for generalized moment problems (see, e.g., \cite{shapiro2001duality,popescu2005semidefinite}). Denote by $\mathcal{A}=\operatorname{co}(\mathcal{P}_0)$ the cone of measures $\mathcal{A}$ generated by the set of probability distributions $\mathcal{P}_0$. Define it dual cone as $\mathcal{A}^*:=\{h \in \mathcal{H} : \int_\Omega h(x) \, \rm{d}\P(x) \geq 0,\  \forall \P \in \mathcal{A}\},$ where $\mathcal{H}$ is the linear space of functions formed by combinations of $g,h_0,\dots, h_m$, and the spaces of functions and measures are paired by the integral operator. We now have the necessary preliminaries to demonstrate strong (conic) duality. Lemma~\ref{lemma:equiv}, in conjunction with assumptions (A1)--(A3), pave the way for us to formulate the main results. 

\begin{theorem}[Strong conic duality]\label{thm1}
Suppose that assumptions \textnormal{(A1)--(A3)} hold. Then, the optimal value of the primal problem \eqref{eq:generalmomentproblem} is finite and equals that of its dual problem
\begin{equation}\label{eq:moregeneraldual}
\begin{aligned}
&\inf_{\lambda_0,\dots,\lambda_{m+1}} &    & \lambda_{m+1}\\
& \textnormal{ s.t. } &      & \sum_{j=0}^m  \lambda_j q_j \leq 0, \\
& &      & \sum_{j=0}^{m+1} \lambda_j h_j(x) - g(x) \1_{\event}(x)   \in \mathcal{A}^{*},
\end{aligned}
\end{equation}
in which $\mathcal{A}^{*}$ is the dual cone of $\mathcal{A}$.
\end{theorem}

\begin{proof}
The assumptions ensure that, for all $\P\in\cP$,
$$
\begin{aligned}
\left|\frac{\E_\P[g(X)\1_{\event}(X)]}{\E_{\P}[\1_{\event}(X)]}\right| &\leq \frac{1}{\epsilon}|\E_\P[g(X)]| \\ &\leq \sup_{\P\in\cP}\frac{1}{\epsilon}|\E_\P[g(X)]| \\ &\leq \frac{1}{\epsilon}\sup_{x\in\Omega}|g(x)| < \infty.\\
\end{aligned}
$$
Hence, instead of \eqref{eq:generalmomentproblem}, it is equivalent to consider \eqref{eq:altern}.
Since the constraints of this problem are linear in the probability distribution $\P$, standard conic duality for generalized moment problems suffices, which holds under the Slater-type condition, as defined in \cite{shapiro2001duality}. Notice that if we substitute $\tau$ with $\lambda_3$ and incorporate the dual problem in the constraint of \eqref{eq:altern}, then \eqref{eq:altern} is equivalent to \eqref{eq:moregeneraldual}. For $\mathcal{A} = \mathcal{M}_+$, the strong conic dual problem \eqref{eq:moregeneraldual} reduces to the semi-infinite LP \eqref{eq:generaldual}.
The result for general cones of measures follows from conic duality arguments as in Theorem~3.1 of \cite{popescu2005semidefinite}. 
\end{proof}

As a consequence, duality enables us to reduce the primal problem, which has infinite-dimensional variables, to a dual problem with $m+1$ variables, but with an infinite number of constraints. These constraints are indexed by the probability measures, i.e., the constraints should hold $\forall\P\in\mathcal{P}_0$. This indexation might turn out to be difficult. However, this difficulty can be greatly reduced if we instead use the generating set $\mathcal{T}$, as shown in \cite{popescu2005semidefinite}. In this case, the dual cone can be reduced to $\mathcal{A}^*=\{h \in \mathcal{H} : \int_\Omega h(x) \  {\rm d}\P(x) \geq 0,\  \forall \P \in \mathcal{T}\}$, and hence, the indexing now only runs over the set of extreme points of $\cP_0$.

Another classical result states that the semi-infinite LP that models a generalized moment problem can be reduced to an equivalent finite-dimensional problem with the same optimal value.
The Richter-Rogosinski Theorem (see, e.g.,~\cite{rogosinski1958moments,smith1995generalized,han2015convex})
states that there exists an extremal distribution for the semi-infinite LP with at most $m+1$ support points. Analogous to the basic solutions for conventional semi-infinite linear programming, we define a basic distribution as a convex combination of extremal distributions of $\cP_0$. As there are $m+1$ moment functions, these basic distributions consist of at most $m+1$ extreme points (i.e., elements of $\mathcal{T}$). We let $\mathcal{D}(\vec{q})$ denote the set of basic distributions that comply with the given moment sequence $\vec{q}$. We can then state the following result, which is an adaptation of the fundamental theorem for convex classes of distributions, but now for the generalized conditional-bound problem.

\begin{theorem}[Fundamental theorem for conditional expectations]\label{thm2}
Consider problem \eqref{eq:generalmomentproblem}. Under assumptions \textnormal{(A1)--(A3)}, 
$$
\sup _{\P \in \cP(\vec{q})} \E_\P[g(X)\mid X\in \event]=\sup _{\P \in \mathcal{D}(\vec{q})} \E_\P[g(X)\mid X\in \event].
$$
Moreover, if the optimal value is attained, then there exists an optimal basic distribution, a convex combination of $m+1$ probability distributions from the generating set $\mathcal{T}$, that achieves this value.
\end{theorem}
\begin{proof}
    Once again, under the stated assumptions, it suffices to consider the equivalent problem \eqref{eq:altern}. Since the constraints of \eqref{eq:altern} are linear in $\P$, standard generalized moment problem results apply. Hence, Theorem~6.1 in \cite{popescu2005semidefinite} is applicable and yields the desired result. This is underpinned by Lemma~\ref{lemma:equiv}, which establishes that the suprema in both \eqref{eq:generalmomentproblem} and \eqref{eq:altern} are attained either by the same extremal distribution or asymptotically achieved by an identical sequence of distributions.
\end{proof}
This theorem states that, even if the bound is not achieved, it is sufficient to consider only the basic feasible distributions in $\mathcal{D}(\vec{q})$ to determine the supremum. This result holds for general convex classes of distributions $\cP$ with the optimal distributions taken as convex combinations of the extremal distributions in $\mathcal{T}$.
We further remark that both theorems also hold even when the supremum of the conditional expectation grows infinitely large. In such cases, there exists a maximizing sequence of probability measures (also taken from the set of basic distributions) for which the conditional expectation diverges. The proof of Lemma~\ref{lemma:equiv} demonstrates that (A2) is sufficient for the equivalence of the two formulations. Assumption (A1) primarily ensures the finiteness of the optimal value. Without (A1) the primal problem may lead to an unbounded solution, rendering the dual problem infeasible by weak duality.
Combined, the concept of weak duality described in Section~\ref{sec:dualprob} and the reduction to the basic distributions generated by $\mathcal{T}$ as proposed in Theorem~\ref{thm2} provide an effective way of solving problem \eqref{eq:generalmomentproblem}, as will be demonstrated in the next section.

\section{Tight bounds for conditional expectations}\label{section3}
In this section, we study several easy examples for the case with $n=1$; that is, $X$ is a random variable conditioned on itself. First, in Section~\ref{sec:meandisp}, we seek the best possible bounds for conditional expectation $\E[X \mid X\geq t]$ when mean-dispersion information, and possibly structural properties of the underlying distribution, are given. We find the tight bounds using primal-dual arguments. In Section~\ref{sec:pricing}, we demonstrate this also works for arbitrary choices of $g(x)$, using an example from the robust pricing literature. Finally, Section~\ref{sec:numericalbounds} shows that, in the univariate setting, tight bounds can be obtained by solving semidefinite programming problems.

\subsection{Simple examples for mean-dispersion information}\label{sec:meandisp}
For the sake of exposition, we concentrate our efforts on the event $\event=\{X\geq t\}$. We thus seek to bound the conditional expectation
$$
\E_\P[g(X) \mid X\geq t]= \frac{\int_\Omega g(x) \1_{\event}(x)\Dist{x}}{\int_\Omega \1_{\event}(x)\Dist{x}},
$$
in which $\1_{\event}(x)$ is the indicator function modeling the occurrence of the event $\{X\geq t\}$ and $\P$ is the underlying probability distribution of which we assume that it lies in the mean-variance ambiguity set, $\cP_{(\mu,\sigma)}$, which contains all distributions that comply with the available mean-variance information. Then, the problem of interest can be stated as 
\begin{equation}\label{eq:primalfractional}
\begin{aligned}
&\sup_{\P(x)\geq0} &  &\frac{\int_\Omega g(x)\1_{\event}(x)\Dist{x}}{\int_\Omega \1_{\event}(x)\Dist{x}}\\
&\text{s.t.} &      & \int_\Omega \Dist{x}=1, \int_\Omega x\, \Dist{x}=\mu,\ \int_\Omega x^2 \, \Dist{x}=(\sigma^2+\mu^2),   
\end{aligned}
\end{equation}
which is a semi-infinite LFP. Through the generalized Charnes-Cooper transformation, introduced in Section~\ref{sec:generalframe},
it is possible to write \eqref{eq:primalfractional} as 
\begin{equation}\label{eq:primallp}
\begin{aligned}
&\sup_{\alpha,\mathbb{Q}(x)\geq0} &  &\int_\Omega g(x) \1_{\event}(x)\dF{x}\\
&\text{  s.t.} &      & \int_\Omega \dF{x}=\alpha, \int_\Omega x \, \dF{x}=\alpha\mu,\ \int_\Omega x^2 \, \dF{x}=\alpha(\sigma^2+\mu^2), \int_\Omega\1_{\event}(x) \dF{x}=1,   
\end{aligned}
\end{equation}
where $\alpha$ is a decision variable denoting the ``worst-case'' probability of $X$ exceeding $t$. The semi-infinite linear programming dual of \eqref{eq:primallp} is given by
\begin{equation}\label{eq:dual}
\begin{aligned}
&\inf_{\lambda_0,\lambda_1, \lambda_2, \lambda_3} &  & \lambda_3\\
&\text{  s.t. } &      & \lambda_0 + \lambda_1 \mu + \lambda_2(\sigma^2+\mu^2)\leq 0, \\
& &      & \lambda_0+\lambda_1 x+\lambda_2 x^2+\lambda_3 \1_{\event}(x) \geq g(x)\1_{\event}(x), \ \forall x\in\mathbb{R},
\end{aligned}
\end{equation}
which, for this specific choice of random event $\event$, results in
$$
\begin{aligned}
&\inf_{\lambda_0,\lambda_1, \lambda_2, \lambda_3} &  & \lambda_3 &\\
&\text{  s.t. } &      & \lambda_0 + \lambda_1 \mu + \lambda_2(\sigma^2+\mu^2)\leq 0, &\\
& &      & \lambda_0+\lambda_1 x+\lambda_2 x^2 \geq 0, \ &\forall x<t, \\
& &      & \lambda_0+\lambda_1 x+\lambda_2 x^2+\lambda_3  \geq g(x), \ &\forall x\geq t. \\
\end{aligned}
$$
Consider the standard conditional expectation (i.e., consider the function $g: x\mapsto x$). We next try to find feasible solutions to the dual problem, and prove optimality by finding primal solutions with matching objective values. Let $t<\mu$. The dual problem of $\sup_{\mathbb{P}\in\mathcal{P}_{(\mu,\sigma)}}\E[X\mid X\geq t]$ is given by
\begin{equation}
\begin{aligned}
&\inf_{\lambda_0,\lambda_1, \lambda_2, \lambda_3} &  & \lambda_3 &\\
&\text{  s.t. } &      & \lambda_0 + \lambda_1 \mu + \lambda_2 (\sigma^2+\mu^2)\leq 0, &\\
& &      & \lambda_0+\lambda_1 x+\lambda_2 x^2 \geq 0, \ &\forall  x<t, \\
& &      & \lambda_0+\lambda_1 x+\lambda_2 x^2   \geq x - \lambda_3, \ &\forall x\geq t. \\
\end{aligned}
\end{equation}

Denote the left-hand sides of the constraints by $M(x)\coloneqq\lambda_0+\lambda_1 x +\lambda_2 x^2$. The function $M(x)$ is dual feasible when it is greater than or equal to 0 for $x<t$ and greater than or equal to $x-\lambda_3$ for $t\leq x \leq b$.

\begin{figure}[h!]
\begin{center}
\begin{tikzpicture}[scale=1,
    declare function={identityx(\x)=0*(\x<=2.5)+(\x-2.5)*(\x>=2.5);}
    ]
    \draw[-latex,thick] (-0.5, 0) -- (6, 0) node[right] {\footnotesize$x$};
    \draw[-latex,thick] (0, -2) -- (0, 3);
    \draw[blue,thick] (2,3) -- (2.7,3) node[right] {\footnotesize$\1_{\{x\geq t\}}(x-\lambda_3)$};
    \draw[dashed] (2,2.5) -- (2.7,2.5) node[right] {\footnotesize$M_1(x)$};
    \draw[dash dot] (2,2.0) -- (2.7,2.0) node[right] {\footnotesize$M_2(x)$};
    \draw[scale=1, domain=0:2, smooth, variable=\x, blue, thick]  plot ({\x}, {0});
    \draw[densely dashed, blue] (2,0)--(2,-2);
    \draw[densely dashed, blue] (2,-2)--(2,-2) node[left]{\footnotesize$t-\lambda_3$};
    \draw[blue] (2,0) circle (1.5pt);
    \filldraw[blue] (2,-2) circle (1.5pt);
    \draw[blue, thick] (2,-2)--(6,2);
    \draw[dotted] (3,0)--(3,0)  node[below left]{\footnotesize$\mu$};
    \draw[dotted] (2,0)--(2,0)  node[below left]{\footnotesize$t^-$};
    \draw[dotted] (3.5,-.5)--(3.5,0)  node[above]{\footnotesize$x_0$};
    \filldraw[black] (3.5,-.5) circle (1.5pt);
    \draw[scale=1, domain=0:5, smooth, variable=\x, dashed]  plot ({\x}, {0.11*(\x-1)*(\x-1)});
    \draw[scale=1, domain=1:5, smooth, variable=\x, dash dot]  plot ({\x}, {0.888889 * (\x)^2 - 5.22222*\x + 6.88889});
\end{tikzpicture}
\end{center}
\centering
\caption{$M_1(x)$ and $M_2(x)$}\label{fig:majorsE[X|X>t]meanvariance}
\end{figure}
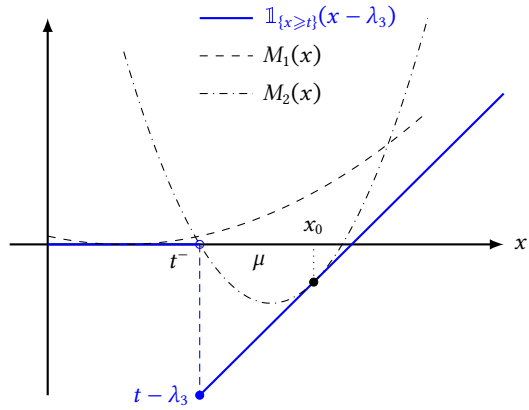

The first dual function, $M_1(x)$, does not admit a feasible solution. Since $\sigma>0$, from $\lambda_0 + \lambda_1\mu + \lambda_2(\sigma^2+\mu^2) \leq 0$ it follows that $M(\mu)=\lambda_0 + \lambda_1\mu + \lambda_2 \mu^2<0$, but for the minimizer $x^*=-\lambda_1/(2\lambda_2)$, $M(x^*)\geq0$. Thus, this case is infeasible. 

The second parabola, $M_2(x)$, does admit a feasible solution. This function coincides with the function $g(\cdot)$ at $t$ and some point $x_0$. Since $M_2(t)=0$, $M_2(x_0)=x_0-\lambda_3$ and $M^{'}_2(x_0)=1$,
\begin{equation}\label{eq:lambdas}
\lambda_0=-\frac{t(\lambda_3(t-2x_0)+x_0^2)}{(t-x_0)^2} ,\ \lambda_1=\frac{t^2+x_0(x_0-2\lambda_3)}{(t-x_0)^2},\ \lambda_2 = \frac{\lambda_3-t}{(t-x_0)^2}.
\end{equation}
Hence, this yields an optimization problem in two variables, $\min_{x_0,\lambda_3} \lambda_3$, with the additional constraint $\lambda_0 + \lambda_1\mu + \lambda_2(\sigma^2+\mu^2) \leq 0$, in which we substitute \eqref{eq:lambdas}. If this constraint is tight, the dual problem can be reduced to
$$
\min_{\lambda_3,x_0}\lambda_3 \equiv \min_{x_0} \frac{(t^2 + x_0^2)\mu - t (x_0^2+\mu^2+\sigma^2)}{(t-\mu)(t-2x_0+\mu)-\sigma^2}.
$$
Optimizing over $x_0$, it then follows that 
$$
x^*_0 = \mu + \frac{\sigma^2}{\mu-t},\ \lambda_3^* = \mu + \frac{\sigma^2}{\mu-t},
$$
with $\lambda_3^*$ a feasible upper bound for $\E[X|X\geq t]$. To prove this bound is optimal, we construct a distribution that (asymptotically) achieves $\lambda_3^*$. From complementary slackness (see, e.g., \cite{smith1995generalized}), we deduce that the candidate distribution has all of its probability mass on two points: $t^-$ and $x_0^*$. Solving the system of moment constraints in \eqref{eq:primalfractional} yields the probabilities
$$
p_{t^-}= \frac{\sigma^2}{\sigma^2 + (\mu-t)^2},\ p_{x_0^*} =  \frac{(\mu-t)^2}{\sigma^2 + (\mu-t)^2}.
$$
Indeed, for this two-point distribution,
$$
\E[X \mid X\geq t] =  \frac{x_0^*\cdot p_{x_0^*}}{\P(X\geq t)} = \frac{x_0^*\cdot p_{x_0^*}}{p_{x_0^*}} = x_0^* = \mu + \frac{\sigma^2}{\mu-t},
$$
ensuring the upper bound is tight. Hence, by weak duality,
\begin{equation}
    \sup_{\mathbb{P}\in\mathcal{P}_{(\mu,\sigma)}} \E[X \mid X\geq t] = \mu + \frac{\sigma^2}{\mu-t}.
\end{equation}
Notice that this bound is not actually attained, but it is achieved in the limit; that is, for $t_k=t-\frac1k$, as $k\to\infty$, this construction of the extremal distribution indeed gives the desired result.
For $t\geq\mu$, it can be shown that
$$
\sup_{\mathbb{P}\in\mathcal{P}_{(\mu,\sigma)}} \E[X \mid X\geq t] = \infty,
$$
since we can construct the following sequence of (maximizing) measures: 
$$
\P_k = \frac{1}{k^2\sigma^2 + 1}\delta_{\mu+k\sigma^2} + \frac{\sigma^2}{\sigma^2 + \frac{1}{k^2}}\delta_{\mu-\frac1k}.
$$
Letting $k\to\infty$ then results in $\E_{\P_k}[X\mid X\geq t]\to\infty$. Taken together, we obtain the following result.

\begin{proposition}\label{prop:meanvar}
For a real-valued random variable $X$ with distribution $\P\in\mathcal{P}_{(\mu,\sigma)}$,  it holds that
    \begin{equation}
    \sup_{\mathbb{P}\in\mathcal{P}_{(\mu,\sigma)}} \E_\P[X \mid X\geq t] =\begin{cases}
    \mu + \frac{\sigma^2}{\mu-t},\quad &\textnormal{for } t<\mu, \\
    \infty, &\textnormal{for } t\geq\mu.
    \end{cases}
\end{equation}
\end{proposition}

A number of interesting observations can be drawn from this result. First, note that the maximizing sequence for the second case, $\{\P_k\}$, converges weakly to $\delta_{\mu}$, which is \emph{not} included in the ambiguity set. 
Notice also that the solution to the second case becomes ``uninformative,'' as it diverges for values of $t\geq\mu$.
Degenerate behavior like this also holds for different ambiguity sets, as we will see in later examples. 
This result confirms tightness of the Mallows and Richter bound for conditional expectations under mean-variance information, stated in \cite{mallows1969inequalities}. Further, notice that the worst-case distribution that achieves the upper bound matches the extremal distribution that yields the Cantelli lower bound for the tail probability, as also shown in \cite{giannakopoulos2020robust}. The authors of the latter work provide a constructive proof of tightness using the extremal distribution that achieves the Cantelli bound. Despite Proposition~\ref{prop:meanvar} being a known result, this is the first time it has been proven through a duality argument that provides immediate insight into the extremal distributions.

The method discussed above can be applied to different types of dispersion information, not only the traditional variance. For example, assume that instead of the variance, we consider the mean absolute deviation from the mean (MAD), $d\coloneqq\E|X-\mu|$, as the measure of dispersion. Let $\mathcal{P}_{(\mu,d)}$ denote the mean-MAD ambiguity set, with the additional constraint that the support of $X$ is (a subset of) the interval $[a,b]$. We can then prove the following result.
\begin{proposition}\label{prop:mad}
For a real-valued random variable $X$ with distribution  $\P\in\mathcal{P}_{(\mu,d)}$,  it holds that
\begin{equation}
    \sup_{\P\in\cP_{(\mu,d)}}\E_\P[X \mid X\geq t]=\begin{cases}
    \mu+\frac{d(\mu-t)}{2(\mu-t)-d},\quad &\textnormal{for } t<\mu-\frac{d(b-\mu)}{2(b-\mu)-d}, \\
    b, &\textnormal{for } t\geq\mu-\frac{d(b-\mu)}{2(b-\mu)-d}.
    \end{cases}
\end{equation}
\end{proposition}
Again, we see that the second case becomes uninformative, as it simply reduces to the robust solution (i.e., it agrees with the upper bound of the support). It is worth noting here the interplay between the size of the ambiguity set and the set that describes the random event/side information. Despite having only a limited number of distributions to choose from, if the realizations of the random variable are limited to too small an interval, the bounds become overly conservative, as an extremal distribution can be constructed for which the moment constraints are satisfied, yet the support point on $\event$ can be made arbitrarily large, bounded, of course, by the upper bound of the support. For the mean-MAD ambiguity set, the extremal distribution also agrees with the distribution attaining the lower bound on the corresponding tail probability \citep{roos2022tight}.

Proposition~\ref{prop:mad} can be generalized further. Assume that the dispersion information is modeled through the expectation of a convex function $d(\cdot)$ of the random variable $X$, defined as $\bar{\sigma}:=\E[d(X)]$.
We can state the following result for such arbitrary convex dispersion measures.
\begin{proposition}\label{prop:meandisp}
Suppose that there exists a solution $x_0^*\in(t,\infty)$ to the equation
$$
\frac{\bar{\sigma} t-\mu  d(t)}{t d(x_0)-x_0 d(t)} +  \frac{\mu  d(x_0) - \bar{\sigma}  x_0}{t d(x_0) - x_0 d(t)} = 1, 
$$
such that the corresponding two-point distribution, with support $\{t,x^*_0\}$, is feasible. Then, for a real-valued random variable $X$ with distribution $\P\in\mathcal{P}_{(\mu,\bar{\sigma})}$,  it holds that
\begin{equation}
    \sup_{\P\in\cP_{(\mu,\bar{\sigma})}}\E_\P[X \mid X\geq t]=
    x_0^*. 
\end{equation}
\end{proposition}
Proposition~\ref{prop:meandisp} covers a wide range of dispersion measures, not limited to only variance and MAD. It also incorporates asymmetric measures of dispersion, such as semivariance, semimean deviations, and partial moments. More generally, it encompasses all dispersion measures that are modeled using piecewise convex functions. Since the $p$-norms on $\mathbb{R}$ are convex, these are naturally included in the category of convex dispersion measures as well. Another notable function that falls into this class is the Huber-loss function, which has been extensively studied in the field of robust statistics.

As explained earlier, using only moment information often leads to overly conservative bounds and pathological worst-case distributions. 
We require additional assumptions about the distribution's shape to sharpen the bounds and avoid the pathological two-point distributions that constitute the worst-case scenario in the previous examples. We next study two such structural properties, i.e., symmetry and unimodality.
The random variable $X$ is said to admit a symmetric distribution about a point $m$ if $\P(X\in[m-x,m]) = \P(X\in[m,m+x])$ for all $x\geq0$. A random variable $X$ has a unimodal distribution with mode $m$ if its distribution function is a concave function on $(-\infty,m]$ and convex on $(m,\infty)$. Both definitions are generalized so that they admit probability distributions that allow for point masses at $m$.
We next consider a distribution that is symmetric about its mean $\mu$ with the values of the mean and variance given. Making use of primal-dual arguments, we obtain the following result.

\begin{proposition}\label{prop:sym}
For a real-valued random variable $X$ with a symmetric distribution $\P\in\cP^{\rm sym}_{(\mu,\sigma)}$, it holds that
\begin{equation}
    \sup_{\P\in\cP^{\rm sym}_{(\mu,\sigma)}}\E_\P[X \mid X\geq t]=\begin{cases}
    \mu + \frac{(\mu-t)\sigma^2}{2(t-\mu)^2-\sigma^2},\quad &\textnormal{for } t<\mu-\sigma, \\
    \mu + \sigma,\quad &\textnormal{for } \mu-\sigma \leq t<\mu, \\
    \infty, &\textnormal{for } t\geq\mu.
    \end{cases}
\end{equation}
\end{proposition}
Observe that the bounds are sharper than the bound in Proposition~\ref{prop:meanvar}, but still vacuous for $t\geq\mu$.
Combining the notions of symmetry and unimodality yields the following, even tighter, bounds:
\begin{proposition}\label{prop:unimod}
For a real-valued random variable $X$ with a symmetric, unimodal distribution $\P\in\cP^{\rm uni}_{(\mu,\sigma)}$, it holds that
\begin{equation}
    \sup_{\P\in\cP^{\rm uni}_{(\mu,\sigma)}}\E_\P[X \mid X\geq t]=\begin{cases}
    \frac{4 \mu  (x^*_0)^3-3 \sigma^2 (t+x^*_0-\mu) (t-x^*_0 + \mu)}{4 (x^*_0)^3-6 \sigma ^2 (t+x^*_0-\mu )},\quad &\textnormal{for } t<\mu -\frac{\sqrt{3}\left(2 \sqrt{2}+1\right)}{7}   \sigma, \\
    \frac{1}{2} \left(\mu + t + \sqrt{3} \sigma \right), &\textnormal{for } \mu -\frac{\sqrt{3}\left(2 \sqrt{2}+1\right)}{7}   \sigma<t<\mu, \\
    \infty, &\textnormal{for } t\geq\mu,
    \end{cases}
\end{equation}
where $x_0^*$ is the real-valued solution to the quartic equation
$$
6\sigma^2 x_0^2 \left(3 (t-\mu )^2 - x_0^2\right)+9 \sigma ^4 (\mu-t-x_0)^2 = 0,
$$
which satisfies the condition $x^*_0\geq\sqrt{3}\sigma$.
\end{proposition}
Although Proposition~\ref{prop:unimod} does not provide a closed-form solution, it does demonstrate the versatility of the primal-dual arguments used to derive it. It further highlights that structural properties can be addressed in conjunction with moment information, even for conditional-bound problems.


\subsection{A robust pricing example}\label{sec:pricing}
To demonstrate the primal-dual approach for an alternative objective function $g(x)$, we next turn our attention to a specific minimax problem. We consider the objective of the robust monopoly-pricing problem; see, e.g., \cite{giannakopoulos2020robust,chen2022distribution,elmachtoub2021value}. This involves evaluating the revenue function
$$
\Pi(z):= \E[p\1_{\{X\geq p\}}] = p\P(X\geq p),
$$
which models the expected revenue that a seller of a single item receives when the price posted is equal to $p$, and the valuation of customers is distributed as $X$.
As in \cite{giannakopoulos2020robust}, we will attempt to minimize the maximum relative regret by posting the minimax selling price $p^*$; that is, we solve
$$
\min_{p}\max_{\mathbb{P} \in \mathcal{P}} \frac{\max_z\E_{\P}[\Pi(z)]}{\E_{\P}[\Pi(p)]}.
$$
We use the “min” and “max” operators only to avoid notational clutter, as it does not imply that the optima are actually attained. \citet{chen2022distribution} present various results for robust monopoly pricing and also consider this relative regret criterion. In \cite{chen2022distribution} an almost identical objective function is considered, namely,
$$
\min_p \max_{\mathbb{P}}\left\{1-\frac{\E_{\P}[\Pi(p)]}{\max_z\E_{\P}[\Pi(z)]}\right\}=1-\max _p \min _{\mathbb{P}} \frac{\E_{\P}[\Pi(p)]}{\max_z\E_{\P}[\Pi(z)]}.
$$
We, however, work with the reciprocal
$$
\min_p \max_{\mathbb{P}} \frac{\max_z\E_{\P}[\Pi(z)]}{\E_{\P}[\Pi(p)]}=\min_p \max_{\mathbb{P}} \max_z \frac{\E_{\P}[\Pi(z)]}{\E_{\P}[\Pi(p)]} = \min_p \max_z \max_{\mathbb{P}} \frac{\E_{\P}[\Pi(z)]}{\E_{\P}[\Pi(p)]},
$$
where we swap the maximization operators to obtain the final equality. As in the proof of Theorem 4 in \cite{chen2022distribution}, it is imperative to solve the semi-infinite optimization problem
\begin{equation}\label{eq:primalprice}
\max_{\mathbb{P} \in \mathcal{P}_{(\mu, \sigma)}} \frac{\mathbb{E}_{\mathbb{P}}[z \1\{X\geq z\}]}{p\mathbb{P}(X \geq p)}.
\end{equation}
Notice that this problem effectively models the conditional expectation $\E[\frac{z}{p} \frac{\1\{X\geq z\}}{\1\{X\geq p\}} \mid X\in \event]$ with the random event $\event=\{X\geq p\}$.
As \citet{chen2022distribution} try to optimize the dual problem directly, their proof requires several lengthy, tedious arguments to obtain the tight bounds. We will simplify their proof using primal-dual arguments as in the previous subsection. We assume that $p<\mu$, as otherwise, it is possible to construct an extremal distribution for which the relative regret ratio diverges; see \cite{giannakopoulos2020robust}.
For $z<p$, the expectation in the numerator will evaluate to 1. Hence,
$$
\max_{\mathbb{P} \in \mathcal{P}_{(\mu, \sigma)}} \frac{\mathbb{E}_{\mathbb{P}}[\frac{z}{p} \1\{X\geq z\}]}{\mathbb{P}(X \geq p)} \leq \frac{z}{p}\max_{\mathbb{P} \in \mathcal{P}_{(\mu, \sigma)}} \frac{1}{\mathbb{P}(X \geq p)} = \frac{z}{p}\frac{1}{\min\limits_{\P\in \mathcal{P}_{(\mu, \sigma)}}\mathbb{P}(X \geq p)}.
$$
To bound the latter, we can use the one-sided version of Chebyshev's inequality (commonly known as Cantelli's inequality). It then follows that
$$
\max_{\mathbb{P} \in \mathcal{P}_{(\mu, \sigma)}} \frac{\mathbb{E}_{\mathbb{P}}[\frac{z}{p} \1\{X\geq z\}]}{\mathbb{P}(X \geq p)} =\frac{z}{p}\frac{\sigma^2 + (\mu-p)^2}{(\mu-p)^2}.
$$
Hence,
$$
\max_{z\leq p} \max_{\P} \frac{\mathbb{E}_{\mathbb{P}}[\frac{z}{p} \1\{X\geq z\}]}{p\mathbb{P}(X \geq p)} = \frac{\sigma^2 + (\mu-p)^2}{(\mu-p)^2}.
$$

For $z\geq p$, it holds that
$$
\frac{\mathbb{E}_{\mathbb{P}}[\frac{z}{p} \1\{X\geq z\}]}{\mathbb{P}(X \geq p)} = \frac{z}{p}\frac{\P(X\geq z)}{\P(X\geq p)} = \frac{z}{p}\frac{\P(X\geq z\cap X\geq p)}{\P(X\geq p)} = \frac{z}{p} \P(X\geq z \mid X\geq p),
$$
and thus, we require a bound for conditional probabilities.
Using primal-dual arguments, we obtain the following result.
\begin{proposition}\label{prop:prob}
Suppose that $z\geq p$ and $p<\mu$. For a nonnegative random variable $X$ with distribution $\P\in\cP_{(\mu,\sigma)}$, it holds that
\begin{equation}
    \sup_{\mathbb{P}\in\mathcal{P}_{(\mu,\sigma)}} \P(X\geq z \,|\, X\geq p) =  \begin{cases}
    \frac{\sigma^2}{(z-\mu)^2 + \sigma^2},\quad &\textnormal{for } z\geq \mu + \frac{2\sigma^2(\mu-p)}{(\mu-p)^2 + \sigma^2},\\
    \left(\frac{(\mu-p)^2+\sigma^2}{\sigma^2+\mu^2-p^2+2z(\mu-p)}\right)^2, &\textnormal{for } \frac{\sigma^2+\mu^2-p\mu}{\mu-p}\leq z \leq \mu + \frac{2\sigma^2(\mu-p)}{(\mu-p)^2 + \sigma^2}, \\
        1,\quad &\textnormal{otherwise. } \\
    \end{cases}
\end{equation}
\end{proposition}
Using these bounds, one can show that
$$
\max_{z\geq p} \sup_{\P} \frac{\mathbb{E}_{\mathbb{P}}[\pi(z,X)]}{p\mathbb{P}(X \geq p)} = \frac{\sigma^2 + \mu(\mu-p)}{p(\mu-p)}.
$$
Hence, combining the above results, the optimal price should solve
$$
\min _{p<\mu} \max \left\{\frac{\sigma^2}{p(\mu-p)}+\frac{\mu}{p}, \frac{\sigma^2}{(\mu-p)^2} + 1\right\}.
$$
It is shown by \cite{giannakopoulos2020robust} that the optimal price is the value of $p$ on which both branches of the max operator agree. Observe that by using primal-dual arguments, we have greatly reduced the number of necessary calculations to obtain the desired result. Contrary to \cite{chen2022distribution}, we work with the primal and dual problems concurrently so as to verify the optimality of the suggested solutions in a more effective manner.
It goes without saying that the proposed duality approach in this section probably also works for other pricing problems with fractional objectives, such as e.g.~the personalized pricing setting in \cite{elmachtoub2021value}.

\subsection{Numerical bounds}\label{sec:numericalbounds}
We next show how the strong dual problem \eqref{eq:generaldual} can be reduced to a semidefinite programming problem for the univariate setting. Consequently, the tight bounds can always be obtained numerically as the solution to a (computationally tractable) optimization problem. The numerical experiments have been conducted in the Julia programming language, using the MOSEK solver \citep{mosek} together with the Julia
packages SumOfSquares.jl and PolyJuMP.jl \citep{weisser2019polynomial}.

Assume the objective function $g(x)$ is piecewise polynomial and the moment constraints are described by the traditional power moments. Then the dual problem can be reduced further to a semidefinite programming problem by applying standard DRO techniques discussed in, for example, \cite{bertsimas2005optimal,bertsimas2002relation,popescu2005semidefinite}. In particular, the univariate moment problem reduces to solving a semidefinite program, provided that the dual-feasible set is semi-algebraic; that is, the dual constraint involves checking whether certain polynomial functions are nonnegative on intervals described by the event set $\event$ and the support $\Omega$. 
It is well known that a univariate polynomial is nonnegative if, and only if, it is a sum of squares.
A classic result then states that the semi-infinite constraint in the dual-feasible set of \eqref{eq:generaldual},
with the support $\Omega$ a possibly infinite interval, can be reduced to a set of linear matrix inequalities (LMIs) of polynomial size in the number of moments $m$ (see, e.g., \cite{nesterov1997structure,bertsimas2005optimal,bertsimas2002relation}). 
Since $g(x)$ is piecewise polynomial, the support of the dual constraints in \eqref{eq:generaldual} can be subdivided into subintervals so that these constraints can be reduced to a set of LMIs.
Generalized moment information can be included if these moments are described by piecewise polynomials in the dual problem.
Examples of piecewise polynomial objective functions are the indicator function $g(x) = \1_{[c,\infty)}(x)$ and the stop-loss function $g(x)=\max\{x-c,0\}$. Both of these functions have several relevant applications in e.g.~finance, insurance and inventory control.

\begin{figure}[h!]
\begin{subfigure}{0.5\textwidth}
    \centering
        \centering
        \includegraphics{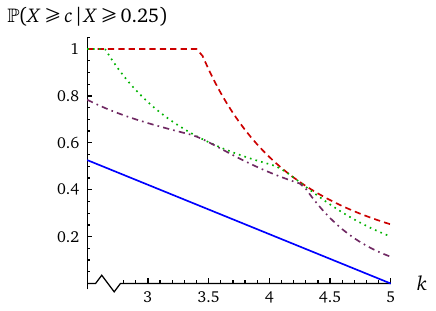}
        \caption{$\event=\{X\geq0.25\}$}
        \label{fig:condproba}
\end{subfigure}%
\begin{subfigure}{0.5\textwidth}
    \centering
        \centering
        \includegraphics{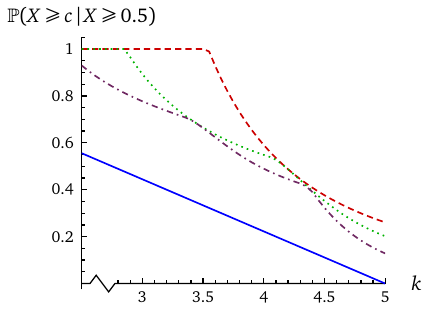}
        \caption{$\event=\{X\geq 0.5\}$}
        \label{fig:condprobb}
\end{subfigure}
\vspace{0.5cm}
\centering
\begin{subfigure}{0.5\textwidth}
    \centering
        \centering
        \includegraphics{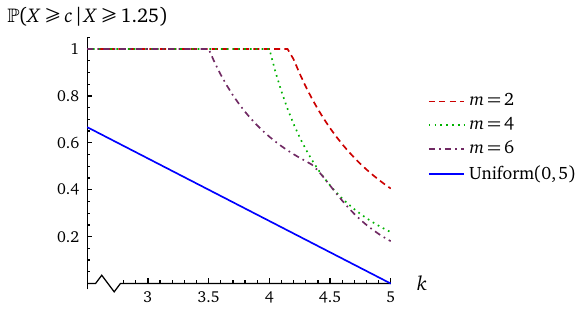}
        \caption{$\event=\{X\geq1.25\}$}
        \label{fig:condprobc}
\end{subfigure}
\vspace{0.1cm}
\caption{Tight bounds for conditional tail probability for ambiguity sets matching the uniform distribution on $[0,5]$}
\label{fig:condprob}
\end{figure}

\begin{figure}[h!]
\begin{subfigure}{0.5\textwidth}
    \centering
        \centering
        \includegraphics{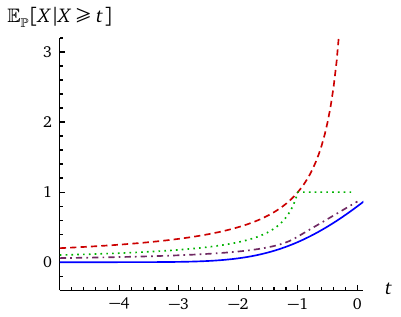}
        \caption{$m=2$}
        \label{fig:normal2m}
\end{subfigure}%
\begin{subfigure}{0.5\textwidth}
    \centering
        \centering
        \includegraphics{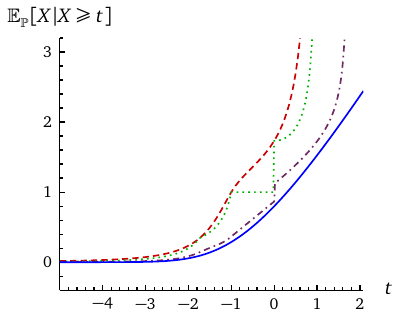}
        \caption{$m=4$}
        \label{fig:normal4m}
\end{subfigure}
\vspace{0.5cm}
\centering
\begin{subfigure}{0.5\textwidth}
    \centering
        \centering
        \includegraphics{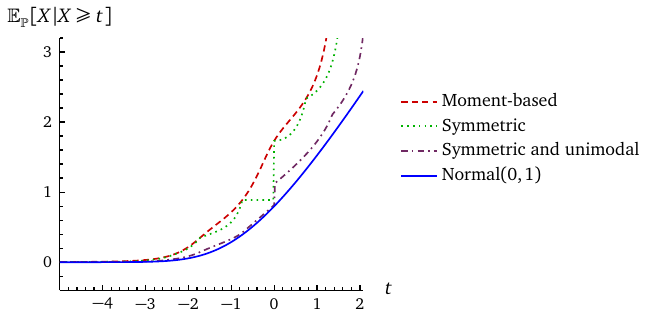}
        \caption{$m=6$}
        \label{fig:normal6m}
\end{subfigure}
\vspace{0.1cm}
\caption{Tight bounds for conditional expectation for ambiguity sets matching the moments and properties of the standard normal distribution}
\label{fig:normalmomentsexample}
\end{figure}

In Figure~\ref{fig:condprob} we provide numerical bounds for the conditional probability $\P(X\geq c\mid X\geq t)$, which corresponds to the piecewise function $g(x) = \1_{[c,\infty)}(x)$. We assume that the ground truth is given by a uniform distribution with support $[0,5]$. We determine the bounds for three types of ambiguity sets, in which the number of available moments varies between $m=2,4$, and 6. Obviously, the bounds become sharper when more moment information is included, but in addition, the uninformative solution becomes less prominent because the size of the ambiguity set reduces. Nevertheless, the uninformative solution becomes more apparent again as the size of $\event$ reduces, as already noted in Section~\ref{sec:meandisp}.

Even if the ambiguity sets are augmented with structural properties, the resulting dual problem can still be reduced to a semidefinite optimization problem.
Consider the setting in which $\cP_0$ contains only symmetric distributions. Using the generator class consisting of symmetric pairs of Dirac measures, the dual problem \eqref{eq:moregeneraldual} reduces to
\begin{equation}\label{eq:dualsymsdp}    
\begin{aligned}
&\inf & &\lambda_3 &\\
&\text{s.t. } &      & \sum_{j=0}^{m} \lambda_{j}q_j\leq 0, \\
&& & \sum_{j=0}^{m+1} \lambda_{j}\left(h_j(\mu-x)+h_j(\mu+x)\right) \geq \\
&&& \quad \quad g(\mu-x)\1_{\event}(\mu-x) + g(\mu+x)\1_{\event}(\mu+x),\, \forall x \geq 0.
\end{aligned}
\end{equation}

Likewise, if we consider symmetric, unimodal distributions (which can be generated by the convex combination of a Dirac measure $\delta_{\mu}$ and rectangular/uniform distributions, i.e., $\delta_{[\mu-z,\mu+z]},\  z>0$), the dual problem becomes 
\begin{equation}\label{eq:dualunimodsdp}    
\begin{aligned}
&\inf & &\lambda_3 &\\
&\text{s.t. } &      & \sum_{j=0}^{m} \lambda_{j}q_j\leq 0, \\
&& & \sum_{j=0}^{m+1} \lambda_{j}\int_{\mu-x}^{\mu+x} h_j(z) {\rm d}z \geq \int_{\mu-x}^{\mu+x} g(z) \1_{\event}(z) {\rm d}z,\, \forall x \geq0, \\
&&& \sum_{j=0}^{m+1} \lambda_{j}h_j(\mu) \geq g(\mu)\1_{\event}(\mu).
\end{aligned}
\end{equation}
Observe that the resulting integral transforms are still piecewise polynomial in $x$. As a consequence, we can reformulate both \eqref{eq:dualsymsdp} and \eqref{eq:dualunimodsdp} as semidefinite programming problems.

Figure~\ref{fig:normalmomentsexample} shows the results for different structural assumptions. Notice that the addition of structural information significantly sharpens the bounds, independently of the available moment information. However, even though the bounds are sharpened with additional information, the bounds still diverge for $t\geq\mu$ when only mean-variance information is available. Still, it is obvious from the figures that this conservatism of the bounds can be mitigated significantly by adding additional moment information.

\section{Distributionally robust optimization with side information}\label{section4}
In this section, we lay the groundwork for a generalized moment-based framework for contextual distributionally robust stochastic optimization.
Section~\ref{sec:contextualDRO} introduces moment-based contextual DRO and a class of ambiguity sets that lead to tractable problems. In Section~\ref{sec:DROmeandisp}, we provide examples of these ambiguity sets for which computationally tractable conic optimization problems can be derived. 
Finally, in Section~\ref{sec:newsvendor}, we present some numerical results for a two-dimensional example.

\subsection{Contextual DRO with mean-dispersion information}\label{sec:contextualDRO}
Given the side information in the form of the event $\vec{X}\in\event$, contextual DRO problems can be stated in general as
\begin{equation}\label{eq:dro}
\inf_{\gvec{\nu}\in\mathcal{V}} \sup_{\mathbb{P}\in\cP} \mathbb{E}_{\mathbb{P}}[f(\gvec{\nu},\vec{X}) \mid \vec{X}\in\event],
\end{equation}
where $f$ is some cost function to be minimized, and $\gvec{\nu}\in\mathcal{V}$ denotes the decision vector with $\mathcal{V}\subseteq\mathbb{R}^p$ a closed, convex set. The probability distribution $\P$ is the joint measure governing $\vec{X}$. Let $\vec{Y}\in\mathcal{Y}\subseteq\mathbb{R}^{n_\vec{y}}$ be a random vector that models the outcome variables that affect the decision problem directly, and let $\vec{Z}\in\mathcal{Z}\subseteq\mathbb{R}^{n_\vec{z}}$ be the covariates (or features) that influence the outcome random variables. Assuming that the supports of $\vec{Y}$ and $\vec{Z}$ are independent, let $\vec{X}=(\vec{Y},\vec{Z})\in\mathcal{Y}\times\mathcal{Z}=:\mathcal{X}$. Henceforth the boldface lowercase characters represent the realizations of the random vectors.
Furthermore, the expectation $\inf_{\gvec{\nu}}\sup_{\P\in\cP}\E_{\P}[v(\gvec{\nu},\vec{X}) \mid \vec{X}\in\event]$ is conditioned  primarily on the information given by the covariates $\vec{Z}$ with $\event_\vec{z}\subseteq\mathcal{Z}$ the information set built from the information on the covariates $\vec{Z}$. Therefore, the side information is described by $\event:=\{\vec{x}=(\vec{y},\vec{z})\in\mathcal{X} : \vec{z}\in\event_{\vec{z}}\}$. This includes the case in which $\event_{\vec{z}}$ is represented by a singleton, which models a particular realization of the covariates. No conditional information is normally included about the outcome variables. Hence, in the remainder of the section, we occasionally use the notation $\E[f(\gvec{\nu},\vec{X}) \mid \vec{Z}\in\event_{\vec{z}}]$ for the conditional expectation given the side information.

We next introduce some additional technical notation tailored to this section of the paper. Boldfaced lowercase characters represent vectors, where the italic character ${x}_k$ denotes the $k$th element of the vector $\vec{x}$ and $\vec{x}^{\top}$ denotes its transpose. Except for the random vectors described above, all boldface uppercase characters represent matrices. For a set $\mathcal{S}$, let $\operatorname{conv}(\mathcal{S})$, $\operatorname{cl}(\mathcal{S})$ and $\operatorname{int}(\mathcal{S})$ denote its convex hull, closure and interior, respectively. For a proper cone $\mathcal{K}\in\mathbb{R}^n$ (i.e., a closed, convex and pointed cone with nonempty interior), the general inequality $\vec{x}\preceq_{\mathcal{K}}\vec{u}$ is equivalent to the set constraint $\vec{u}-\vec{x}\in\mathcal{K}$, while the strict variant $\vec{x}\prec_{\mathcal{K}}\vec{u}$ expresses that $\vec{u}-\vec{x}\in\operatorname{int}(\mathcal{K})$. A function $\vec{h}:\mathbb{R}^n\mapsto\mathbb{R}^m$ is called $\mathcal{K}$-convex if $\vh(\theta\vx_1 + (1-\theta)\vx_2) \preceq_{\mathcal{K}} \theta\vh(\vx_1) +(1-\theta) \vh(\vx_2)$ for all $\vec{x}_1,\vec{x}_2$ and $\theta\in[0,1]$. We use $\mathcal{K}^*$ to denote the dual cone of $\mathcal{K}$, given by $\mathcal{K}^*=\{\vec{u}:\vec{u}^{\top}\vec{x},\ \forall \vec{x}\in\mathcal{K}\}$ with $\vec{u}^{\top}\vec{x}$ the appropriate inner product. The set $\mathbb{S}_+^n$ represents the cone of symmetric positive semidefinite matrices in $\mathbb{R}^{n\times n}$. Finally, for matrices $\vec{A},\vec{B}$, we use $\vec{A}\preceq \vec{B}$ to abbreviate the relation $\vec{A}\preceq_{\mathbb{S}_+^n} \vec{B}$.

In order to derive solvable reformulations of \eqref{eq:contextualDRO},  we shall impose the following conditions:
\begin{itemize}
    \item[(C1)] The side information $\event$ is a closed, convex set with $\inf_{\P\in\cP}\P(X\in\event)>0$, and the support set $\mathcal{X}=\mathbb{R}^{n}$.
    \item[(C2)] The dispersion is modelled by ($\mathcal{D}$-)convex epigraph functions (see \cite{wiesemann2014distributionally,hanasusanto2017ambiguous}). 
    \item[(C3)] The function $f(\gvec{\nu},\vec{x})$ can be represented as
    $$
    f(\gvec{\nu},\vec{x}) = \max_{l\in\mathcal{L}}\{f_l(\gvec{\nu},\vec{x})\},
    $$
    in which $\mathcal{L}$ is a set of indices and the auxiliary functions $f_l$ are of the form
    $$
    f_l(\gvec{\nu},\vec{x}) = \vec{s}_l(\gvec{\nu})^\top \vec{x}  + t_l(\gvec{\nu})
    $$
    where for all $l\in\mathcal{L}$, $\vec{s}_l(\cdot)\in\mathbb{R}^n$  and $t_l(\cdot)\in\mathbb{R}$ are some affine functions of $\gvec{\nu}$.
    \item[(C4)] The ambiguity set $\cP$ satisfies the Slater condition. 
\end{itemize}
\vspace{.25cm}
The rationale behind the first part of condition (C1) is twofold: (i) it enables the use of robust optimization methods to reformulate the model, and (ii) it guarantees that the side event has nonzero measure. There is no necessity for the second part of the condition (closedness and convexity would suffice), as it merely serves to facilitate the mathematical exposition in this section.

The second condition states that the dispersion function has an epigraph that can be described through convex cones. We denote the dispersion function by $\vec{d}: \mathbb{R}^n \mapsto \mathbb{R}^m$, and assume it admits a $\mathcal{D}$-epigraph that is conic representable, with $\mathcal{D}\subseteq\mathbb{R}^m$ a proper cone, meaning that the set $\{(\vec{x}, \vec{u}) \in \mathbb{R}^n \times \mathbb{R}^m: \vec{d}(\vec{x}) \preceq_\mathcal{D} \vec{u}\}$ can be described with conic inequalities, possibly using cones other than $\mathcal{D}$ and auxiliary variables.
See \cite{ben2001lectures} for a comprehensive introduction to conic representations. 
The {epigraphic} mean-dispersion ambiguity set can now be defined as
\begin{equation}\label{eq:epigraphicmeandisp}
\cP_{(\gvec{\mu},\gvec{\sigma})}=\{\P\in \mathcal{P}_0(\mathcal{X}):\, \E_{\P}[\vec{X}]=\gvec{\mu}, \, \E_{\P}[\vec{d}(\vec{X})] \preceq_{\mathcal{D}} \boldsymbol{\sigma}\},
\end{equation}
where $\cP_0$ is the set of probability distributions with support $\mathcal{X}$, the vector  $\gvec{\mu} \in \mathbb{R}^n$ represents the mean value of the random vector $\vec{X}$, and $\boldsymbol{\sigma} \in \mathbb{R}^m$ is an upper bound on the expected value of the dispersion measure $\E[\vec{d}(\vec{X})]$. 
Although the mean-dispersion ambiguity set $\cP_{(\gvec{\mu},\vd)}$ may seem simple, there are numerous practically relevant ambiguity sets that can be recovered by selecting appropriate dispersion functions $\vd(\cdot)$. For example, setting $\vd=\vec{0}$ yields the mean-support ambiguity set, while setting $\vd(\vec{X})=(\vec{X}-\gvec{\mu})(\vec{X}-\gvec{\mu})^{\top}$ enables modeling the correlation structures between the elements of the random vector $\vec{X}$. Other ($\mathcal{D}$-)convex dispersion measures include the Huber loss function, mean absolute deviations, any norm $\norm{\cdot}$ on $\mathbb{R}^n$, and the other convex dispersion measures that were applicable to Proposition~\ref{prop:meandisp}. We will elaborate on some of these dispersion measures in the next subsection.

The third condition states that the objective function $f$ is a convex, piecewise-affine function of the uncertainty $\vx$, for all $\vnu\in\mathcal{V}$, and the decision vector $\vnu$, for all $\vx\in\mathcal{X}$.  We will focus on piecewise-affine objective functions as these are more than sufficient to capture several interesting models. For example, they can capture max operators as in the newsvendor model, as well as the Conditional-Value-at-Risk, which is frequently used to optimize financial portfolios with risk-averse investors. The requirement on the objective function is not strictly necessary and can be relaxed to a much richer class of functions $f(\cdot,\cdot)$ that are convex in $\gvec{\nu}$ for all $\vx$, and concave in $\vec{x}$ for all admissible $\gvec{\nu}$. For a detailed discussion on computational tractability, we refer the interested reader to \cite{wiesemann2014distributionally}.


The dual problem of the inner maximization problem of \eqref{eq:dro}, in the multivariate case, is given by
\begin{equation}\label{eq:multidual}
\begin{aligned}
&\inf_{\lambda_0,\vlambda_1,\vlambda_2,\lambda_3} &    & \lambda_{3}\\
&\text{s.t. } &      &  \lambda_0 + \vlambda_1^\top \gvec{\mu} + \vlambda_2^\top \gvec{\sigma} \leq 0, \\
& &      & \lambda_0 + \vlambda_1^\top \vx + \vlambda_2^\top \vd(\vx) \geq (f(\vnu,\vx)-\lambda_{3}) \1_{\event}(\vx), \ \forall \vx\in\mathcal{X},
\end{aligned}
\end{equation}
with $\lambda_0,\lambda_3\in\mathbb{R}$, $\gvec{\lambda}_1\in\mathbb{R}^n$ and $\gvec{\lambda}_2\in\mathcal{D}^*$.
We assume that condition (C4) holds so that $\sup_{\mathbb{P}\in\cP} \mathbb{E}_{\mathbb{P}}[f(\gvec{\nu},\vec{X}) \mid \vec{X}\in\event]$ is strongly dual to \eqref{eq:multidual}. To be more specific, the Slater condition is given by $\gvec{\mu}\in\operatorname{int}(\mathcal{X})$ and $\vd(\gvec{\mu})\prec_{\mathcal{D}}\gvec{\sigma}$.
The semi-infinite constraint in \eqref{eq:multidual} can be amended using standard robust optimization methods. This yields the following result.

\begin{theorem}[Contextual DRO with mean-dispersion information]\label{thm3}
Let $\P$ be a member of the mean-dispersion ambiguity set $\cP_{(\gvec{\mu},\gvec{\sigma})}$. If conditions \textnormal{(C1)--(C4)} hold, then the objective value of the contextual DRO problem \eqref{eq:dro} coincides with the optimal value of the semi-infinite LP
\begin{equation}\label{eq:contextualDRO}
\begin{aligned}
&\inf_{\gvec{\nu},\lambda_0,\vlambda_1,\vlambda_2,\lambda_3} &    & \lambda_{3}\\
&\textnormal{s.t. } &         &\lambda_0 + \vlambda_1^\top \gvec{\mu} + \vlambda_2^\top \gvec{\sigma} \leq 0, \\
&&&\lambda_0 + \vlambda_1^\top \vec{x} + \vlambda_2^\top \vu \geq 0, \quad &\forall (\vx,\vu) \in \overline{\mathcal{C}},\\
  &&&\lambda_0 + \vlambda_1^\top \vec{x} + \vlambda_2^\top \vu  +\lambda_{3}  \geq \vec{s}_l(\gvec{\nu})^\top \vec{x}  + t_l(\gvec{\nu}), \quad &\forall (\vec{x},\vu)\in {\mathcal{C},\  }\forall l\in\mathcal{L}, \\
  &&& \gvec{\nu}\in\mathcal{V},\ \lambda_0\in\mathbb{R},\  \gvec{\lambda}_1\in\mathbb{R}^n,\  \gvec{\lambda}_2\in\mathcal{D}^*,\ \lambda_3\in\mathbb{R}. &
\end{aligned}
\end{equation}
Here,
\begin{equation}\label{eq:convhull}
\begin{aligned}
{\mathcal{C}} &:= \{(\vec{x},\vu)\in\mathbb{R}^n\times\mathbb{R}^m:\vec{x}\in\event, \ \vd(\vec{x})\preceq_{\mathcal{D}}\vu\}, \\ \overline{\mathcal{C}} &:= \operatorname{conv}\left(\{(\vec{x},\vu)\in\mathbb{R}^n\times\mathbb{R}^m:\vec{x}\in\complevent, \ \vd(\vec{x})=\vu\}\right),
\end{aligned}
\end{equation}
in which $\complevent:=\operatorname{cl}(\mathbb{R}^n\backslash \event)$. Moreover, under additional regularity conditions, problem \eqref{eq:contextualDRO} admits a reformulation as a finite-dimensional conic optimization problem.
\end{theorem}
\begin{proof}
The dual problem of the inner maximization problem is given by
\begin{equation}\label{eq:multidual2}
\begin{aligned}
&\inf_{\lambda_0,\vlambda_1,\vlambda_2,\lambda_3} &    & \lambda_{3}\\
&\text{subject to} &      &  \lambda_0 + \vlambda_1^\top \gvec{\mu} + \vlambda_2^\top \gvec{\sigma} \leq 0, \\
& &      & \lambda_0 + \vlambda_1^\top \vx + \vlambda_2^\top \vd(\vx) \geq \1_{\event}(\vx)\big(f(\vnu,\vx)-\lambda_{3}\big), \ \forall \vx\in\mathbb{R}^n.
\end{aligned}
\end{equation}
By decomposing the semi-infinite constraints using the definition of the indicator function, we obtain two semi-infinite constraints,
\begin{equation}\label{eq:DeMorganconstraints}
\begin{aligned}
    &\lambda_0 + \vlambda_1^\top \vx + \vlambda_2^\top \vd(\vx) \geq 0, \quad &\forall \vx\in\mathbb{R}^n\backslash\event, \\
  &\lambda_0 + \vlambda_1^\top \vx + \vlambda_2^\top \vd(\vx)  +\lambda_{3}  \geq f(\vnu,\vx), \quad &\forall \vx\in\event,
\end{aligned}
\end{equation}
in which $\mathbb{R}^n\backslash\event$ is the complement of $\event$. Since $\vlambda \in \mathcal{D}^*$ and $\vd(\vx)$ is $\mathcal{D}$-convex by assumption, it holds that $ \lambda_0 + \vlambda_1^\top \vx + \vlambda_2^\top \vd(\vx)$ is a convex function of $\vx$ (see, e.g., \cite[p.~110]{boyd2004convex}) and, a fortiori, continuous in $\vx$. From a standard continuity argument, it then follows that we are allowed to replace the complement with its closure, $\complevent$.  Since $f(\vnu,\vx)$ is a convex, piecewise affine function by condition (C3), \eqref{eq:DeMorganconstraints} is equivalent to
$$
\begin{aligned}
&\lambda_0 + \vlambda_1^\top \vx + \vlambda_2^\top \vd(\vx) \geq 0, \quad &\forall \vx\in\complevent, \\
&\lambda_0 + \vlambda_1^\top \vx + \vlambda_2^\top \vd(\vx)  +\lambda_{3}  \geq f_l(\vnu,\vx), \quad &\forall \vx\in\event,\ \forall l\in\mathcal{L}.
\end{aligned}
$$
Then, by lifting the nonlinearity in the uncertainty to the uncertainty set, we obtain the robust counterparts
\begin{equation}\label{eq:DeMorganconstraints2}
\begin{aligned}
    &\lambda_0 + \vlambda_1^\top \vx + \vlambda_2^\top \vu \geq 0, \quad &\forall (\vx,\vu) :\vx\in\complevent,\  \vd(\vx)=\vu,\\
  &\lambda_0 + \vlambda_1^\top \vx + \vlambda_2^\top \vu  +\lambda_{3}  \geq f_l(\vnu,\vx), \quad &\forall (\vx,\vu) :\vx\in\event,\  \vd(\vx)=\vu,\ \forall l\in\mathcal{L}.
\end{aligned}
\end{equation}
As the constraints are linear in the uncertain parameters, we can equivalently use the convex hull of the uncertainty sets
\begin{equation}\label{eq:DeMorganconstraints3}
\begin{aligned}
    &\lambda_0 + \vlambda_1^\top \vec{x} + \vlambda_2^\top \vu \geq 0, \quad &\forall (\vx,\vu) \in \operatorname{conv}\left(\{(\vec{x},\vu):\vec{x}\in\complevent, \ \vd(\vec{x})=\vu\}\right),\\
  &\lambda_0 + \vlambda_1^\top \vx + \vlambda_2^\top \vu  +\lambda_{3}  \geq f_l(\vnu,\vx), \quad &\forall (\vec{x},\vu)\in\textnormal{conv}\big(\{(\vec{x},\vu):\vec{x}\in\event, \vd(\vec{x})=\vu\} \big),\ \forall l\in\mathcal{L},
\end{aligned}
\end{equation}
From conditions (C1) and (C2), it follows that the convex hull of the uncertainty set for the second set of robust counterparts is equivalent to $\mathcal{C}$. It is important to note that for the first set of semi-infinite constraints, the convex hull also forms a convex set.
As the robust counterparts in \eqref{eq:DeMorganconstraints3} constitute an (infinite) intersection of halfspaces with respect to the dual variables $\lambda_0,\vlambda_1,\vlambda_2,\lambda_3$ and the decision vector $\gvec{\nu}$,  it is deduced that the feasible set is convex and \eqref{eq:contextualDRO} is a convex optimization problem. This in turn implies that problem \eqref{eq:contextualDRO} can be phrased as a conic optimization problem, provided the convex sets ${\mathcal{C}}$ and $\overline{\mathcal{C}}$ allow for tractable conic reformulations that meet a Slater condition. To substantiate the second claim, in Appendix~\ref{app:conics} we recast \eqref{eq:contextualDRO} as a finite-dimensional conic program under these supplementary conditions.
This completes the proof.
\end{proof}

The difficulty now lies in reformulating the semi-infinite constraints (or robust counterparts) in the dual problem by constructing explicit expressions for the convex hulls $\mathcal{C}$ and $\overline{\mathcal{C}}$. To address this, robust optimization techniques for nonlinear types of uncertainty can be used; see, for example,  \cite{yanikouglu2019survey,ben2015deriving} for further details. It may be insightful to note here that optimization over $\mathcal{C}$ is generally computationally tractable, provided that the event set and epigraph are (tractable) conic representable, but the reformulation of the semi-infinite constraint that involves the complement is typically not---since it entails optimizing a convex function over a nonconvex set. To simplify the problem, it is reasonable to consider only polyhedral event sets $\event$. This assumption, although sacrificing generality, enhances tractability of the first semi-infinite constraint. Using De Morgan's laws, the semi-infinite constraint involving $\complevent$ can be decomposed into optimization over halfspaces. We consider such an event set in the next subsection. 
It also seems noteworthy to mention here that the distribution-free analysis of \eqref{eq:contextualDRO} shares many similarities with the literature based on uncertainty quantification \citep{hanasusanto2015distributionally,hanasusanto2017ambiguous}, and distributionally robust convex optimization \citep{wiesemann2014distributionally}.
However, in contrast to Theorem~5 in \cite{wiesemann2014distributionally}, we apply a lifting argument when solving the dual problem, rather than during the construction of the ambiguity set. In the next section, we show that Theorem~\ref{thm3} leads to computationally tractable models for appropriate choices of $\cP$ and $\event$.


\subsection{Some examples for mean-dispersion information}\label{sec:DROmeandisp}
For the sake of exposition, we limit our attention to two types of ambiguity sets, which are analogous to Propositions~\ref{prop:meanvar} and Proposition~\ref{prop:mad} in Section~\ref{sec:meandisp}. Further, for the sake of simplicity, we assume that the event set is defined by a halfspace; that is,
$$
\Xi_{\vec{z}} = \{\vec{z}\in\mathbb{R}^{n_\vec{z}}:\vec{c}^{\top}\vec{z}\leq \bar{c}\},
$$
with $\vec{c}\in\mathbb{R}^{n_{\vec{z}}}$ and $\bar{c}\in\mathbb{R}$.
Since $\mathcal{Y}=\mathbb{R}^{n_{\vec{y}}}$, $\event$ is unrestricted in the outcome space. We have chosen this specific setup so that, in the remainder of this section, we can obtain computationally tractable conic reformulations. In this context, “computationally tractable” means that our problems can be formulated as linear, conic-quadratic or, to a lesser degree, semidefinite programs so that we are able to use mature, off-the-shelf solvers for conic optimization. The derivations of these conic programs are provided in the Appendix.

We first construct a Chebyshev-type ambiguity set \citep{delage2010distributionally,vandenberghe2007generalized}, which allows us to impose conditions on the covariance matrix of the random vector $\vec{X}$. Let $\E[\vec{X}]=\gvec{\mu}$ denote the mean vector, and define the dispersion measure as $\vec{d}(\vx)=(\vx-\gvec{\mu})(\vx-\gvec{\mu})^\top$. We identify $\mathcal{D}$ with the cone of positive
semidefinite matrices. The Chebyshev ambiguity set then consists of all distributions with mean $\gvec{\mu} \in \mathbb{R}^n$ and covariance matrix bounded above by $\boldsymbol{\Sigma} \in \mathbb{S}_{+}^n$. It can be defined as
$$
\mathcal{P}_{(\gvec{\mu},\gvec{\Sigma})}=\left\{\mathbb{P} \in \mathcal{P}_0\left(\mathbb{R}^n\right): \mathbb{E}_{\mathbb{P}}[\vec{X}]=\boldsymbol{\mu},\  \mathbb{E}_{\mathbb{P}}\left[(\vec{X}-\boldsymbol{\mu})(\vec{X}-\boldsymbol{\mu})^{\top}\right] \preccurlyeq \boldsymbol{\Sigma}\right\}.
$$
When we consider this ambiguity set in conjunction with the half-space event set, the constraints in \eqref{eq:contextualDRO} can be described by LMIs. Hence, Theorem~\ref{thm3} yields the following result.


\begin{corollary}[Chebyshev ambiguity set]\label{cor:Chebyshev}
Suppose conditions \textnormal{(C1)}--\textnormal{(C4)} are satisfied. Let $\event_{\vec{z}}$ be defined by a halfspace. Then, for $\cP=\mathcal{P}_{(\gvec{\mu},\gvec{\Sigma})}$, the contextual DRO problem \eqref{eq:contextualDRO} can be reformulated as a semidefinite optimization problem. 
\end{corollary}


Alternatively, the MAD
can be used as dispersion measure \citep{postek2018robust}. Let $\vec{m}\in\mathbb{R}^n$ represent some center point, in our case the mean. Assume that we have bounds for the componentwise mean deviations $\mathbb{E}[|\vec{X}-\vec{m}|]$ and the pairwise mean deviations $\mathbb{E}[|({X}_i\pm X_j) -(m_i\pm m_j)|]$, which are given by $\mathbf{f} \in \mathbb{R}^{n^2}$. This information results in the ambiguity set
$$
\begin{aligned}
\mathcal{P}_{(\vec{m},\vec{f})}=\left\{\mathbb{P} \in \mathcal{P}_0\left(\mathbb{R}^n):\right. \mathbb{E}_{\mathbb{P}}[\vec{X}]=\vec{m}, \ \E[|X_i-m_i|]\leq f_{i,i},\forall i,\  \mathbb{E}[|({X}_i\pm X_j) -(m_i\pm m_j)|]\leq f_{i,j}, \forall i\neq j\right\}.
\end{aligned}
$$
For this ambiguity set, the constraints in \eqref{eq:contextualDRO} are representable as linear inequalities. Therefore, Theorem~\ref{thm3} leads to the following result.

\begin{corollary}[MAD ambiguity set]\label{cor:MAD}
Suppose that conditions \textnormal{(C1)}--\textnormal{(C4)} hold. Let $\event_{\vec{z}}$ be defined by a halfspace. Then, for $\cP=\mathcal{P}_{(\vec{m},\vec{f})}$, the contextual DRO problem \eqref{eq:contextualDRO} can be reformulated as a linear optimization problem. 
\end{corollary}
The precise mathematical models are relegated to Appendix~\ref{app:conics}.
Although we have obtained computationally tractable models, it is not immediately clear how the ambiguity sets and side information interact, or under which conditions the contextual DRO problem reduces to its robust counterpart, $\inf_{\gvec{\nu}}\sup_{\vec{y}}f(\gvec{\nu},\vec{y})$, as discussed in Section~\ref{section3}. Related to this, we would like to mention the broad class of nested ambiguity sets that were introduced in \cite{wiesemann2014distributionally}, which encompass several distance-based ambiguity sets as special types of generalized-moment ambiguity sets. The reason for this is that distance-based ambiguity sets
can be defined by a finite number of (conditional) expectation constraints based on generalized moments \citep{chen2020robust}. The distance-based ambiguity sets are particularly interesting as they provide an explicit way to relate the “sizes” of ambiguity sets and event sets. This makes it possible to quantify when the contextual DRO problem becomes “uninformative.”
For an excellent discussion on the interplay between a distance-based ambiguity set based on optimal transport and the size of the event set, see \cite{nguyen2021robustifying}. As our framework applies to generalized moments, it is possible to extend Theorem~\ref{thm3} to include nested and distance-based ambiguity sets. 
Furthermore, as discussed in Section~\ref{section3}, we can obtain tighter bounds by imposing structural properties on the base ambiguity set $\cP_0$. However, both extensions
would entail delving into many technical details. As these might detract from the main focus of this expository section, we leave them to the avid reader.

\subsection{Newsvendor example}\label{sec:newsvendor}
To gain a better understanding of our contextual DRO framework, we now provide a numerical illustration of our models using a well-known problem in OR, namely the (single-item) newsvendor model. In contrast to the traditional model, we pose the problem in the form $\E_{\P}[f(\gvec{\nu},\vec{X})\mid \vec{Z}\in\event_{\vec{z}}]$ to allow for the inclusion of side information. For further insights, please refer to \cite{ban2019big} and the references therein, which discuss the data-driven newsvendor model with feature information. 

Suppose that the newsvendor trades in a single product. Before observing the product demand $D$, the newsvendor places an order for $q$ units of the product. Once the demand is realized, the newsvendor sells all available stock. Any unsatisfied demand results in backorders incurring a penalty cost of $p$ per unit, and any leftover inventory is penalized with a holding cost of $h$ per unit. The newsvendor aims to minimize the total holding and penalty costs, as a function of the order quantity $q\geq0$ and the demand variable $D$,
\begin{equation}\label{eq:newsobj}
\begin{aligned}
    C(q,D) &= h(q-D)^+ + p(D-q)^+\\
    &= h(q-D) + (h+p)(D-q)^+,
\end{aligned}
\end{equation}
with $(x)^+:=\max\{x,0\}$. It is evident that the cost function $C$ is piecewise affine, consistent with condition (C3).
The key distinction from the traditional setting lies in the newsvendor's ability to observe a covariate $Z$, which influences the demand variable $D$, prior to making the ordering decision.
Hence, the contextual DRO problem can be posed as
$$
\inf_{q\geq0} \sup_{\P\in\mathcal{P}_{(\gvec{\mu},\gvec{\Sigma})}} \E_\P[C(q,D)\mid Z\in\event].
$$
The primary objective is to minimize the total costs incurred by the newsvendor while incorporating the information provided about the feature, namely, that $Z\in\event_{\vec{z}}$.


\begin{figure}[h!]
\begin{subfigure}{0.5\textwidth}
    \centering
        \includegraphics{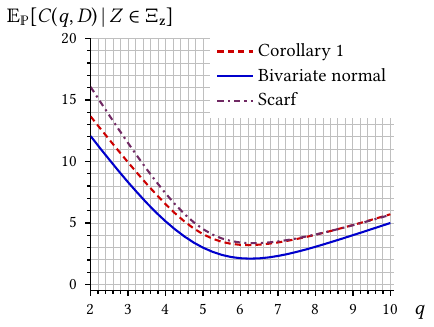}
        \caption{$\event_{\vec{z}}=\{Z\geq1\}$, $\varrho=0$}
        \label{fig:news0}
\end{subfigure}%
\begin{subfigure}{0.5\textwidth}
    \centering
        \includegraphics{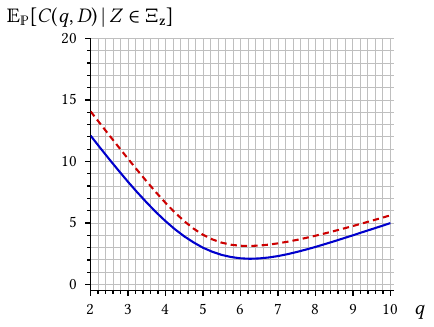}
        \caption{$\event_{\vec{z}}=\{Z\geq1\}$, $\varrho=0.95$}
        \label{fig:news00}
\end{subfigure}
\vspace{0.25cm}
\begin{subfigure}{0.5\textwidth}
    \centering
        \includegraphics{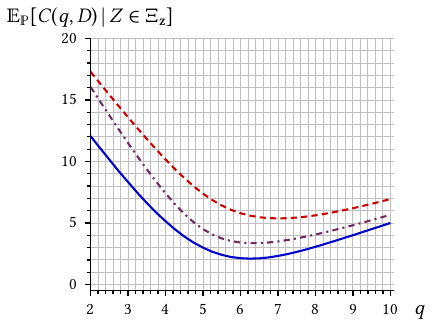}
        \caption{$\event_{\vec{z}}=\{Z\geq4\}$, $\varrho=0$}
        \label{fig:news000}
\end{subfigure}%
\begin{subfigure}{0.5\textwidth}
    \centering
        \includegraphics{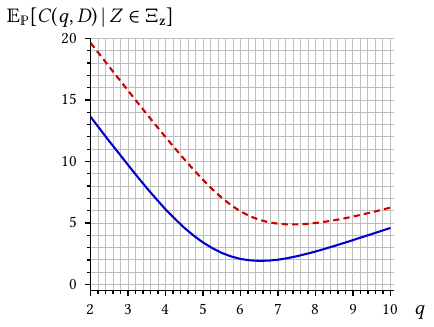}
        \caption{$\event_{\vec{z}}=\{Z\geq4\}$, $\varrho=0.95$}
        \label{fig:news0000}
\end{subfigure}
\vspace{0.2cm}
\caption{Newsvendor cost bounds for different order quantities with dependent and independent demand}
\label{fig:newsbounds}
\end{figure}

To account for the dependence between the demand variable $D$ and the feature $Z$, we introduce a covariance matrix to construct the mean-variance ambiguity set. We assume the ground truth is given by a bivariate normal distribution, characterized by its mean vector $\gvec{\mu}=(5,5)$ and covariance matrix $\gvec{\Sigma}$. The diagonal elements of this matrix are $\sigma^2_{1,1}=2.25,\ \sigma^2_{2,2}=1$, while the off-diagonal elements are determined by the correlation coefficient $\varrho$, representing the dependence between the demand and the feature. We distinguish between two levels of correlation: no correlation, which is represented by $\varrho=0$, and high positive correlation, which is represented by $\varrho=0.95$. Further, we let the holding cost $h=1$, and the penalty cost $p=5$. As for the feature's size, we consider two events: $\{Z\geq1\}$ and $\{Z\geq4\}$. To provide context, think of $Z$ as being demand for an alternative item with similar traits, but which is already partially observed during the planning period.

In Figure~\ref{fig:newsbounds}, the numerical results of the described setting are illustrated, offering several intriguing insights. The plot includes Scarf's bound \citep{Scarf1958} for the zero-correlation setting. Here it is important to note that this bound is valid only for these settings. This is because if the true underlying distribution is a bivariate normal distribution, then the demand $D$ and feature $Z$ are independent when $\varrho=0$. Therefore, a distributionally robust bound that considers only the demand information is sufficient. For the case with zero correlation, Scarf's bound and our bound perform similarly when the event set is of moderate size. However, as the size of the event set decreases, we observe Scarf's bound outperforms our bound. The reason behind this is that the worst-case distribution resulting from optimization over $\cP_{(\boldsymbol{\mu},\boldsymbol{\Sigma})}$ still permits higher orders of dependency. As a result, the value of $Z$ continues to have a significant impact on the bound for the conditional expectation.
Furthermore, it becomes evident that the differences between the conditional-expectation bounds and the true values primarily depend on the size of the event set, rather than that of the correlation coefficient. This is to be expected and aligns with our prior discussions on this topic in previous sections.
Finally, it is essential to keep in mind that the bounds for the conditional expectations continue to diverge as the side information set $\{Z\geq z_0\}$ decreases in size. The only practical way to reduce this conservatism seems to be the inclusion of more information in the ambiguity set.


\section{Conclusions and outlook}\label{section5}
Let us conclude. 
This paper presents a novel framework for bounding conditional expectations that, in contrast to generalized moment-bound problems, can explicitly incorporate side information into the semi-infinite formulation. The key idea is to use a simple transformation to reduce the resulting semi-infinite fractional problem to a semi-infinite LP. The corresponding dual problem highly resembles that of a generalized moment problem, but it involves an additional constraint that models the conditioning on this random event. Fortunately, this slight increase in complexity does not seem to affect significantly the computational tractability of the resulting models. The generalized conditional-bound framework can be used to obtain univariate bounds for different ambiguity sets and general objectives (for, e.g., pricing) through the use of primal-dual arguments. Moreover, it serves as the foundation for a moment-based contextual DRO framework that can be applied to stochastic optimization problems with side information.

We finally mention several potentially interesting avenues for further research. First, it seems of interest to find more applications for the univariate bounds, such as the robust monopoly-pricing problem. Second, alternative applications for the contextual DRO framework, discussed in Section~\ref{section4}, can be investigated. Moreover, it would be beneficial to expand our findings to nested ambiguity sets, as this class of ambiguity includes the distance-based ambiguity sets, which offer a more direct way to answer questions about when a solution becomes “uninformative,” i.e., for which instances the DRO problem reduces to a robust optimization problem. Conducting a comprehensive complexity analysis for the nested ambiguity sets and different types of side information and objective function structures also seems a worthwhile topic to explore. In conclusion, our proposed framework provides a promising approach for bounding conditional expectations while incorporating side information. We anticipate that the suggested directions for future research will contribute to the development and applicability of this framework.


\ACKNOWLEDGMENT{
The author would like to thank Johan van Leeuwaarden, Pieter Kleer and Bas Verseveldt for communicating their versions of Propositions~\ref{prop:meanvar},~\ref{prop:mad} and~\ref{prop:meandisp}, which they derived independently through a more direct approach, essentially solving the primal problem. The author would also like to thank Dick den Hertog for proofreading the first draft of this manuscript.}
 

\medskip


\begin{appendix}
\section{Proofs}\label{app:proofs}
\begin{proof}[Proof of Lemma~\ref{lemma:equiv}.]
Assume, without loss of generality, that the moment constraints are consistent so that $\cP$ is nonempty. Otherwise, both problems are infeasible with optimal value $-\infty$, as per the conventional definition. Let $\{\tilde{\tau},\{\tilde{\P}_k\}_{k\geq1}\}$  and $\{\tau^*,(\tau^*,\{\mathbb{P}_k^*\}_{k\geq1})\}$ be the optimal values, and optimal solutions (maximizing sequences), of \eqref{eq:condmomentprob} and \eqref{eq:altern}, respectively. We can assume maximizing sequences without loss of generality. If the optimal value is exactly achieved by an extremal distribution $\P^*$, we can pick a sequence such that $\lim\limits_{k\to\infty}\P^*_k = \P^*$.
To prove the result, it is sufficient to show that
\begin{equation}\label{eq:reltoshow}
\tau^* = \lim_{k\to\infty}\frac{\E_{\P^*_k}[g(X)\1_{\event}(X)]}{\E_{\P^*_k}[\1_{\event}(X)]} = \lim_{k\to\infty}\frac{\E_{\tilde{\P}_k}[g(X)\1_{\event}(X)]}{\E_{\tilde{\P}_k}[\1_{\event}(X)]} =  \sup_{\P\in\cP} \frac{\E_\P[g(X)\1_{\event}(X)]}{\E_\P[\1_{\event}(X)]} =: \tilde{\tau}.
\end{equation}
The final equality can be attributed to the optimality of $\{\tilde{\P}_k\}$. Note that all values are finite by assumption.
Notice further that
$$
\frac{\E_\P[g(X)\1_{\event}(X)]}{\E_\P[\1_{\event}(X)]} \leq \tilde{\tau},
$$
for all $\P\in\cP$. Rewriting and taking the supremum over $\cP$, the inequality above is equivalent to
$$
\sup_{\P\in\cP} \E_\P[g(X)\1_{\event}(X) - \tilde{\tau}\1_{\event}(X)] \leq 0,
$$
which implies $\tilde{\tau}$ is feasible to problem \eqref{eq:altern}. 
Since $\tau^*$ is a feasible solution to \eqref{eq:altern}, it holds that
\begin{equation}\label{eq:alternconstr}
\sup_{\P \in \mathcal{P}} \mathbb{E}_\P\left[g(X)\1_{\event}(X) -\tau^* \1_{\event}(X)\right] \leq 0.
\end{equation}
Recall that the supremum in \eqref{eq:alternconstr} is achieved by the sequence $\{\mathbb{P}_k^*\}$. Rewriting \eqref{eq:alternconstr}, we obtain
$$
\lim_{k\to\infty}\E_{\mathbb{P}_k^*}[g(X)\1_{\event}(X) - \tau^*\1_{\event}(X)] \leq 0 \iff \tau^* \geq \lim_{k\to\infty}\frac{\E_{\mathbb{P}_k^*}[g(X)\1_{\event}(X)]}{\E_{\mathbb{P}_k^*}[\1_{\event}(X)]},
$$
from which the first identity in \eqref{eq:reltoshow} follows by optimality of $\tau^*$ to problem \eqref{eq:altern}. Since $\tilde{\tau}$ is optimal to \eqref{eq:condmomentprob},
\begin{equation}\label{eq:tautilde}
\tau^* = \lim_{k\to\infty}\frac{\E_{\mathbb{P}_k^*}[g(X)\1_{\event}(X)]}{\E_{\mathbb{P}_k^*}[\1_{\event}(X)]}  \leq  \lim_{k\to\infty}\frac{\E_{\tilde{\P}_k}[g(X)\1_{\event}(X)]}{\E_{\tilde{\P}_k}[\1_{\event}(X)]} =\tilde{\tau}.
\end{equation}
We next show that \eqref{eq:alternconstr} also implies that
\begin{equation}\label{eq:taustar}
\sup_{\P\in\cP} \frac{\E_\P[g(X)\1_{\event}(X)]}{\E_\P[\1_{\event}(X)]}\leq\tau^*.
\end{equation}
That is, the sequence $\{\P^*_k\}$ is also an optimal solution of problem \eqref{eq:condmomentprob}. For the sake of contradiction, assume that for an arbitrary $\epsilon>0$, there exists a sequence of probability measures $\{\tilde{\P}_k\}$ such that
$$
\frac{\E_{\tilde{\P}_k}[g(X)\1_{\event}(X)]}{\E_{\tilde{\P}_k}[\1_{\event}(X)]} > \tau^* + \epsilon,
$$
as $k$ grows sufficiently large. Fixing $\tilde{\P}_k$ for such $k$, we obtain
$$
\E_{\tilde{\P}_k}[g(X)\1_{\event}(X) - \tau^*\1_{\event}(X)] > \epsilon\E_{\tilde{\P}_k}[\1_{\event}(X)].
$$
For $\E_{\tilde{\P}_k}[\1_{\event}(X)]=\tilde{\P}_k(X\in\event)>0$ and $\epsilon>0$, this inequality contradicts \eqref{eq:alternconstr}. Moreover, if $\P_k(X\in\event)=0$, the fractional objective function would be ill-defined, also yielding a contradiction.
Hence, it follows from combining \eqref{eq:tautilde} and \eqref{eq:taustar} that the optimal values of \eqref{eq:condmomentprob} and \eqref{eq:altern} agree, and there exists a sequence of probability distributions that achieves this optimal value in both problems jointly.
This completes the proof. 
\end{proof}

\begin{proof}[Proof of Proposition~\ref{prop:mad}.] Replacing the second-moment function $x^2$ by $|x-\mu|$ and substituting $d$ for $(\sigma^2+\mu^2)$ in \eqref{eq:dual} yields the dual problem for $\sup_{\mathbb{P}\in\mathcal{P}_{(\mu,d)}}\E[X|X\geq t]$,
\begin{equation}
\begin{aligned}
&\inf_{\lambda_0,\lambda_1, \lambda_2, \lambda_3} &  & \lambda_3 &\\
&\text{subject to} &      & \lambda_0 + \lambda_1 \mu + \lambda_2 d\leq 0, &\\
& &      & \lambda_0+\lambda_1 x+\lambda_2 |x-\mu| \geq 0, \ &\forall a\leq x<t, \\
& &      & \lambda_0+\lambda_1 x+\lambda_2 |x-\mu|   \geq x - \lambda_3, \ &\forall t \leq x \leq b. \\
\end{aligned}
\end{equation}

Denote the left-hand sides of the second and third constraints by $D(x)\coloneqq\lambda_0+\lambda_1 x +\lambda_2 |x-\mu|$. The function $D(x)$ is dual feasible when it is greater than or equal to 0 for $x<t$ and greater than or equal to $x-\lambda_3$ for $t\leq x \leq b$. First, we consider the case $t<\mu$. To solve the dual problem, we shall consider three cases for the shape of the dual function $D(x)$, as illustrated in Figure~\ref{fig:majorsE[X|X>t]meanMAD}.

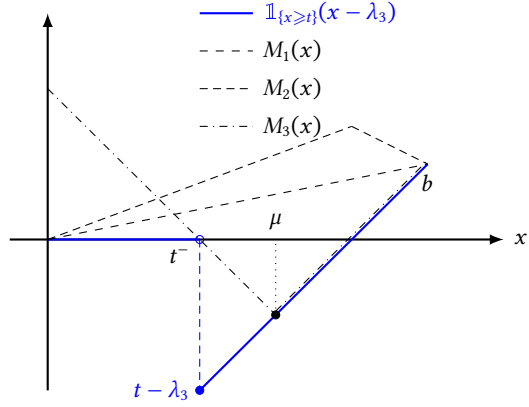
\begin{figure}[h!]
\begin{center}
\begin{tikzpicture}[scale=1,
    declare function={identityx(\x)=0*(\x<=2.5)+(\x-5)*(\x>=2.5);}
    ]
    \draw[-latex,thick] (-.5, 0) -- (6, 0) node[right] {\footnotesize$x$};
    \draw[-latex,thick] (0, -2) -- (0, 3);
    \draw[blue,thick] (2,3) -- (2.7,3) node[right] {\footnotesize$\1_{\{x\geq t\}}(x-\lambda_3)$};
    \draw[dashed] (2,2.5) -- (2.7,2.5) node[right] {\footnotesize$M_1(x)$};
    \draw[densely dashed] (2,2.0) -- (2.7,2.0) node[right] {\footnotesize$M_2(x)$};
    \draw[dash dot] (2,1.5) -- (2.7,1.5) node[right] {\footnotesize$M_3(x)$};
    \draw[scale=1, domain=0:2, smooth, variable=\x, blue, thick]  plot ({\x}, {0});
    \draw[densely dashed, blue] (2,0)--(2,-2);
    \draw[densely dashed, blue] (2,-2)--(2,-2) node[left]{\footnotesize$t-\lambda_3$};
    \draw[blue, thick] (2,-2)--(5,1);
    \filldraw[black] (3,-1) circle (1.5pt);
    \draw[blue] (2,0) circle (1.5pt);
    \filldraw[blue] (2,-2) circle (1.5pt);
    \draw[dotted] (3,-1)--(3,0) node[above]{\footnotesize$\mu$};
    \draw[dotted] (2,0)--(2,0)  node[below left]{\footnotesize$t^{-}$};
    \draw[dashed] (0,.0)--(5,1) node[below]{\footnotesize$b$};
    \draw[densely dashed] (0,0)--(4,1.5);
    \draw[densely dashed] (4,1.5)--(5,1);
    \draw[dash dot] (0,2)--(3,-1);
    \draw[dash dot] (3,-0.95)--(5,1.05);
\end{tikzpicture}
\end{center}
\centering
\caption{$D_1(x)$, $D_2(x)$, and $D_3(x)$}\label{fig:majorsE[X|X>t]meanMAD}
\end{figure}

First, we discuss the case where $D(x)$ is a straight line (i.e.,~$\lambda_2=0$). If $\lambda_1=0$, then $D_1(x)$ is a horizontal line, which is dual feasible because, for a suitable choice of $\lambda_3$, this function lies above $\1_{\{x\geq t\}}(x-\lambda_3)$. Note that for this solution to satisfy $\lambda_0 + \lambda_1 \mu + \lambda_2 d\leq 0$, the constant $\lambda_0=0$. Since we are minimizing $\lambda_3$ (or maximizing $-\lambda_3$), we push the function $x-\lambda_3$ upward until it hits the horizontal line, so this choice for the dual function yields $b$ as the optimal objective value. The dual function $D_1(x)$ cannot have a positive slope because this would imply $\lambda_0+\lambda_1\mu>0$. 
Note that a line with negative slope $\lambda_1<0$ will only increase $\lambda_3$. Hence, the horizontal line with $\lambda_0=0$ is the best feasible option. A primal solution that attains this value is the distribution with support $\left\{\frac{b (d-2 \mu )+2 \mu ^2}{-2 b+d+2 \mu },b  \right\}$ and probabilities
$$
p_1 = 1 - \frac{d}{2(b-\mu)},\quad p_2 = \frac{d}{2(b-\mu)}.
$$

Using complementary slackness, we argue that the second case, $D_2(x)$, can be omitted. From Figure~\ref{fig:majorsE[X|X>t]meanMAD}, observe that the corresponding primal solution is supported on the values $a,b$. However, we cannot construct a two-point distribution that, in general, satisfies the moment constraints. Therefore, this second case does not provide a useful solution from which we can obtain a tight bound.

For the third case (wedge), let $D_3(x)$ coincide with $\1\{x\geq t\}(x-\lambda_3)$ at $x=t$, $\mu$ and $b$. Choosing $D_3(x)$ in this particular way, the dual variables satisfy
$$
\lambda_0=\frac{(t+\lambda_3)\mu-2t\lambda_3}{2(t-\mu)}, \ \lambda_1=\frac{\lambda_3+t-2\mu}{2(t-\mu)}, \ \lambda_2=\frac{t-\lambda_3}{2(t-\mu)}.
$$
Substituting these values into $\lambda_0 + \lambda_1 \mu + \lambda_2 d\leq0$, we obtain
$$
-\lambda_3+\mu+\frac{d(\lambda_3-t)}{2(\mu-t)}\leq0,
$$
in which the left-hand side is a decreasing (linear) function of $\lambda_3$ for $t<\mu-\frac{d(b-\mu)}{2(b-\mu)-d}$. Since we are minimizing with respect to $\lambda_3$, we choose $\lambda_3$ such that equality is attained. Hence, $\lambda_3^*=\mu+\frac{d(\mu-t)}{2(\mu-t)-d}$. The distribution with support $\left\{t^-, \lambda^*_3\right\}$, and respective probabilities
$$
p_1 = \frac{d}{2(\mu-t)},\quad p_2 = 1-\frac{d}{2(\mu-t)},
$$
achieves $\lambda^*_3$ asymptotically.
Combining the two feasible cases, and ensuring these bounds are tight by constructing primal feasible solutions that (asymptotically) achieve these bounds, we arrive at the desired result.
\end{proof}

\begin{proof}[Proof of Proposition~\ref{prop:meandisp}.]
In this setting, the dual is given by
\begin{equation}\label{eq:dualmeandisp}
\begin{aligned}
&\inf_{\lambda_0,\lambda_1, \lambda_2, \lambda_3} &  & \lambda_3 &\\
&\text{subject to} &      & \lambda_0 + \lambda_1 \mu + \lambda_2 d\leq 0, &\\
& &      & \lambda_0+\lambda_1 x+\lambda_2 d(x) \geq 0, \ &\forall x<t, \\
& &      & \lambda_0+\lambda_1 x+\lambda_2 d(x)   \geq x - \lambda_3, \ &\forall x\geq t. \\
\end{aligned}
\end{equation}
Define $V(x):=\lambda_0+\lambda_1 x+\lambda_2 d(x)$. Analogous to the variance and MAD settings, a candidate solution is the majorant $V(x)$ that touches at $x=t^-$, and is tangent to $x-\lambda_3$ at a point $x_0>t$ that will be determined a posteriori. Using these insights, we solve the following system of equations to determine $p_{t},p_{x_0}$ and the dual variables $\lambda_0,\lambda_1,\lambda_2,\lambda_3$ (with $x_0$ fixed):
$$
\begin{aligned}
& p_{t} t + p_{x_0} x_0 = \mu,\ p_{t} d(t) + p_{x_0} d(x_0) = \bar{\sigma}, \\
& \lambda_0 + \lambda_1 t + \lambda_2 d(t) = 0,\  \lambda_0 + \lambda_1 x_0 + \lambda_2 d(x_0) = x_0 - \lambda_3 \\
& \lambda_1 + \lambda_2 d'(x_0) = 1,\ \lambda_0 + \lambda_1 \mu + \lambda_2 (\mu^2 + \sigma^2) = 0,
\end{aligned}
$$
where the first line contains the moment constraint, the second and third line fix the shape of $V(x)$, and finally, we assume that the dual constraint $\lambda_0 + \lambda_1 \mu + \lambda_2 (\mu^2 + \sigma^2) \leq 0$ is tight.
The derivative $d'(x)$ is assumed to be the right derivative, in order to allow for nondifferentiable dispersion functions.
Notice that the resulting dual solution is always feasible, since $V(x)$ is convex ($\lambda_2\geq0$ is a necessary condition for dual feasibility), and therefore, $V(x)\geq (x-\lambda_3)\1_{\event}(x),\ \forall x$, by the constraints $V'(x_0)=1$ and $V(t)=0$. Solving the equations, one obtains
$$
\begin{aligned}
&p_{t} = \frac{\bar{\sigma} t-\mu  d(t)}{t d(x_0)-x_0 d(t)},\ p_{x_0}= \frac{\bar{\sigma}  x_0-\mu  d(x_0)}{x_0 d(t)-t d(x_0)},\ \lambda_3 = \frac{-(t-\mu ) \left(d(x_0)-x_0 d'(x_0)\right)-\mu  d(t)+\bar{\sigma} t}{(t-\mu ) d'(x_0)-d(t)+\bar{\sigma}}, \\
\end{aligned}
$$
and
$$
\begin{aligned}
&\lambda_0 = \frac{\mu  d(t)-\bar{\sigma}  t}{(t-\mu ) d'(x_0)-d(t)+\bar{\sigma} },\ \lambda_1 = \frac{\bar{\sigma} -d(t)}{(t-\mu ) d'(x_0)-d(t)+\sigma },\ \lambda_2 = \frac{t-\mu }{(t-\mu ) d'(x_0)-d(t)+\bar{\sigma} }.
\end{aligned}
$$
To guarantee strong duality, we choose $x_0$ such that 
$$
x_0 = \frac{-(t-\mu ) \left(d(x_0)-x_0 d'(x_0)\right)-\mu  d(t)+\bar{\sigma} t}{(t-\mu )d'(x_0)-d(t)+\bar{\sigma}} = \lambda_3  \iff \frac{(t-x_0)\bar{\sigma} + (x_0-\mu)d(t) + (\mu-t)d(x_0)}{(t-\mu )d'(x_0)-d(t)+\bar{\sigma}} = 0,
$$
and the normalization constraint $p_t+p_{x_0} = 1$ hold. Both conditions are equivalent to 
$$
(t-x_0)\bar{\sigma} + (x_0-\mu)d(t) + (\mu-t)d(x_0) = 0.
$$
Consequently, the $x^*_0$ that follows from solving
$$
\frac{\bar{\sigma} t-\mu  d(t)}{t d(x_0)-x_0 d(t)} +  \frac{\bar{\sigma}  x_0-\mu  d(x_0)}{x_0 d(t)-t d(x_0)} = 1,
$$
is optimal. Hence, the claim follows.
\end{proof}

\begin{proof}[Proof of Proposition~\ref{prop:sym}.]
The class of symmetric pairs of Dirac measures (i.e., $\delta_{\mu-x},\ \delta_{\mu+x},\  x\geq0$) generates the set of symmetric distributions about $\mu$. From Theorem~\ref{thm1}, it follows that the dual problem is given by
\begin{equation}\label{eq:dualsym}
\begin{aligned}
&\inf_{\lambda_0,\lambda_1, \lambda_2, \lambda_3} &  & \lambda_3\\
&\text{subject to } &      & \lambda_0 + \lambda_1 \mu + \lambda_2(\sigma^2+\mu^2)\leq 0, \\
& &      & 2\lambda_0+ 2\lambda_1 \mu+2\lambda_2 (x^2 + \mu^2)+\lambda_3 \1_{\event}(\mu-x)+\lambda_3 \1_{\event}(\mu+x)\\
&& &\qquad \geq (\mu-x)\1_{\event}(\mu-x) + (\mu+x)\1_{\event}(\mu+x), \ \forall x\geq0.
\end{aligned}
\end{equation}
The last constraint can be reduced to
$$
\begin{aligned}
2\lambda_0 + 2\lambda_1 \mu + 2\lambda_2 (x^2 + \mu^2)  \geq -2\lambda_3 + 2\mu, \quad & \forall 0\leq x < \mu-t, \\
2\lambda_0 + 2\lambda_1 \mu + 2\lambda_2 (x^2 + \mu^2) \geq x + \mu - \lambda_3, \quad & \forall x \geq \mu-t.
\end{aligned}
$$
Notice that $\1_{\event}(\mu+x)=1,\ \forall x\geq0$, since it is assumed that $t<\mu$.
Define the quadratic function $M^{\rm sym}(x):= 2\lambda_0 + 2\lambda_1 \mu + 2\lambda_2 (x^2 + \mu^2)$. We suggest two possible solutions for the dual problem. The first solution, denoted as $M_1^{\text{sym}}(x)$, touches $-2\lambda_3 + 2\mu$ at $x=0$ and $x + \mu - \lambda_3$ at $\mu-t$. The second solution, $M_2^{\text{sym}}(x)$, is a quadratic function that touches $x + \mu - \lambda_3$ at an optimal point $x_0$ that is unknown a priori. We further postulate that in both dual solutions, the constraint $\lambda_0 + \lambda_1 \mu + \lambda_2 (\mu^2 + \sigma^2) \leq 0$ is tight. In the interest of space, we omit the figure, but it is easily verified that the suggested solutions are dual feasible. The corresponding primal solutions follow from complementary slackness and are the pairs of Dirac measure $\delta_{\mu-x},\ \delta_{\mu+x}$ in which for $x$ we substitute the points at which the dual function coincides with the right-hand sides of the constraints in \eqref{eq:dualsym}.

The dual variables which correspond to $M_1^{\text{sym}}(x)$ are
$$
\lambda_1 = \frac{-\lambda_0+\frac{(t-\mu)\left(\mu^2+\sigma^2\right)}{2(t-\mu)^2-\sigma^2}}{\mu},\ 
\lambda_2 = \frac{\mu-t}{2(t-\mu)^2-\sigma^2},\  
$$
yielding
$$
\lambda^*_3 = \frac{2(t-\mu)^2 \mu-t \sigma^2}{2(t-\mu)^2-\sigma^2}
$$
as our guess for the optimal value of the dual problem. The proposed solution is feasible for the dual problem since $M_1^{\text{sym}}(x)$ is convex (as $\lambda_2\geq0$) and tangent to $-2\lambda_3 + 2\mu$ at $x=0$, and further some straightforward calculations show that the derivative of $M_1^{\text{sym}}(x)$ at $x=\mu-t$ is greater than 1, so that $M_1^{\text{sym}}(x)\geq x+\mu-\lambda^*_3,\ \forall x\geq\mu-t$. For the primal probabilities, it follows from the variance constraint
$$
(1-p)\mu^2 + \frac{1}{2}p t^2+\frac{1}{2} p (2 \mu -t)^2=\mu^2 +\sigma^2,
$$
that
$$
p = \frac{\sigma ^2}{(t-\mu )^2}.
$$
Hence,
$$
\E[X \mid X\geq t] = \frac{\frac{\sigma ^2 }{2 (t-\mu )^2} (2 \mu -t) + \left(1-\frac{\sigma ^2}{(t-\mu )^2}\right)\mu }{\frac{\sigma ^2}{2 (t-\mu )^2} + \left(1-\frac{\sigma ^2}{(t-\mu )^2}\right)} = \frac{2 \mu  (t-\mu )^2-\sigma ^2 t}{2 (t-\mu )^2-\sigma ^2} = \lambda^*_3,
$$
so that these are the optimal primal-dual solutions by weak duality.
For the second case, we determine the candidate support point $x_0$ first. From the moment constraints and the fact that the solution should be symmetric about $\mu$, we find that 
$
x^*_0 = \sigma
$
and therefore, the primal candidate is given by the distribution $\frac12\delta_{\mu-\sigma} + \frac12\delta_{\mu+\sigma}$.
The second dual solution then yields
$$
\lambda_1 = -\frac{4 \lambda_0 \sigma +\mu^2+\sigma^2}{4 \mu  \sigma },\ \lambda_2 = \frac{1}{4\sigma},\ \lambda^*_3 = \mu +\sigma.
$$
Thus,
$$
\E[X \mid X\geq t] = \mu + x^*_0 = \mu +\sigma = \lambda^*_3.
$$
For $t\geq\mu$, there exists a sequence of measures supported on $\{\mu,\mu-k,\mu+k\}$ that is feasible in the primal, and for which the conditional expectation diverges, as $k\to\infty$. This feasible sequence is given by
$$
\P_k = \left(1-\frac{\sigma ^2}{k^2}\right)\delta_{\mu} + \frac12 \frac{\sigma ^2}{k^2}\delta_{\mu-k} + \frac12 \frac{\sigma ^2}{k^2}\delta_{\mu+k} .
$$
It is then easily verified that $\lim_{k\to\infty}\E_{\P_k}[X\mid X\geq t]$ diverges. Combining the cases above, while checking for feasibility of the primal solutions, the claim follows.
\end{proof}

\begin{proof}[Proof of Proposition~\ref{prop:unimod}.]
Symmetric, unimodal distributions with mode $\mu$ can be generated by rectangular/uniform distributions, possibly including a Dirac measure at $\mu$ (i.e., $\delta_{[\mu-z,\mu+z]},\, z\geq0$). From Theorem~\ref{thm1}, it follows that the dual problem is given by
\begin{equation}\label{eq:dualunimod}
\begin{aligned}
&\inf_{\lambda_0,\lambda_1, \lambda_2, \lambda_3} &  & \lambda_3\\
&\text{subject to } &      & \lambda_0 + \lambda_1 \mu + \lambda_2(\sigma^2+\mu^2)\leq 0,  \\
& &      & \int_{\mu-x}^{\mu+x}\lambda_0+\lambda_1 z+\lambda_2 z^2 {\rm d}z \geq \int_{\mu-x}^{\mu+x} (z- \lambda_3)\1_{\event}(z) \, {\rm d}z, \ \forall x>0, \\
& &      & \lambda_0+\lambda_1 \mu+\lambda_2 \mu^2 \geq (\mu-\lambda_3). 
\end{aligned}
\end{equation}
After computing the integral on the left-hand side of the penultimate constraint, one obtains
$$
\begin{aligned}
\frac{2}{3} x \left(3 \lambda_0+3 \mu  (\lambda_1+\lambda_2 \mu )+\lambda_2 x^2\right) \geq \int_{\mu-x}^{\mu+x} (z- \lambda_3)\1_{\event}(z) \, {\rm d}z, \ \forall x>0.
\end{aligned}
$$
For the right-hand side, we distinguish two cases so that we can split the semi-infinite constraint into two sets, resulting in the system of inequalities
\begin{equation}\label{eq:dualconstraintsuni}
\begin{aligned}
\frac{2}{3} x \left(3 \lambda_0+3 \mu  (\lambda_1+\lambda_2 \mu )+\lambda_2 x^2\right)  \geq 2x(\mu-\lambda_3), \quad & \forall 0 < x < \mu-t, \\
\frac{2}{3} x \left(3 \lambda_0+3 \mu  (\lambda_1+\lambda_2 \mu )+\lambda_2 x^2\right)  \geq -\frac12 (t-x-\mu)(t+x-2\lambda_3+\mu), \quad & \forall x \geq \mu-t.
\end{aligned}
\end{equation}
Again, we can make an educated guess for an optimal dual solution, and using weak duality, prove optimality by constructing a matching primal solution using the complementary slackness property. The dual solution is now characterized by a third-order polynomial function, $M^{\rm uni}(x):= \frac{2}{3} x \left(3 \lambda_0+3 \mu  (\lambda_1+\lambda_2 \mu )+\lambda_2 x^2\right) $. 
We show through primal-dual reasoning that there are merely two feasible options for the extremal distribution. Using this insight, we optimize the primal problem directly by plugging in the candidate form of the extremal distribution. Notice that as $M^{\rm uni}(x)$ needs to be convex in order to be dual feasible, there cannot exist a tangent point on the interval $[0,\mu-t)$ because otherwise $M^{\rm uni}(x)$ would need to intersect $2x(\mu-\lambda_3)$. Furthermore, there can exist only one tangent point $x=x^*_0$ at which $M^{\rm uni}(x)$ coincides with the quadratic function $-\frac12 (t-x-\mu)(t+x-2\lambda_3+\mu)$, as $M^{\rm uni}(x)$ is a cubic function. 
By complementary slackness, the corresponding extremal distribution is then given by the mixture of a Dirac measure at $\mu$ and a uniform distribution on the interval $[\mu-x^*_0, \mu+x^*_0]$, for the first case, or a uniform distribution on $[\mu-\sqrt{3}\sigma, \mu+\sqrt{3}\sigma]$ for the second. Indeed, the latter uniform distribution is the only one that is feasible for the primal problem. From these observations, it follows that the primal problem can be reduced to a finite-dimensional (nonconvex) optimization problem. The objective function of the primal problem can be rewritten as 
\begin{equation}\label{eq:finiteredobj}
\E[X \mid X\geq t] = \frac{(1-p)\mu + p\int_{t}^{\mu+x_0} \frac{z}{2x_0}{\rm d}z}{(1-p)+p\int_{t}^{\mu+x_0} \frac{1}{2x_0}{\rm d}z} = \frac{4x_0 \mu - p(t+x_0-\mu)(t-x_0+\mu)}{4x_0 - 2p(t+x_0-\mu)}.
\end{equation}
From the variance constraint
$$
(1-p)\mu^2 + p\frac{(\mu-x_0)^2+(\mu-x_0)(\mu+x_0)+(\mu+x_0)^2}{3}=\sigma^2+\mu^2,
$$
it follows that $p=\frac{3\sigma^2}{x_0^2}$. In order to be a probability, it should hold that $p\leq1$, and hence $x_0\geq\sqrt{3}\sigma$. The variable $p$ can be eliminated from \eqref{eq:finiteredobj}, yielding the optimization problem
\begin{equation}\label{eq:finitered}
\max_{x_0}\left\{\frac{4 \mu  (x_0)^3-3 \sigma ^2 (-\mu +t+x_0) (\mu +t-x_0)}{4 (x_0)^3-6 \sigma ^2 (-\mu +t+x_0)} : x_0\geq\sqrt{3}\sigma\right\}. 
\end{equation}
From standard arguments, it follows that the maximum of \eqref{eq:finitered} must be attained at a critical point of the objective function or at the boundary of the feasible region. To arrive at the first case, we need to solve the first-order condition
$$
6\sigma^2 x_0^2 \left(3 (t-\mu )^2 - x_0^2\right)+9 \sigma ^4 (\mu-t-x_0)^2 = 0,
$$
which is a depressed quartic equation. It can be shown, after some tedious algebra, that there exists only one real-valued solution which is greater than $\sqrt{3}\sigma$ with the range of $t$ as given in the assertion.
The second case corresponds to the boundary of the feasible region for which $p=1$. As a consequence, the optimal tangent point $x_0^* = \sqrt{3}\sigma$. It is easy to verify that this solution yields the second case.
Finally, for the third case, $t\geq\mu$, we construct a maximizing sequence $\{\P_k\}$ for which $\E[X\mid X\geq t]$ diverges. To this end, consider
$$
\P_k = \left(1-\frac{3\sigma^2}{k^2}\right) \delta_{\mu} + \frac{3\sigma^2}{k^2} \delta_{[\mu-k,\mu+k]},
$$
which is feasible for the primal problem. Then
$$
\E_{\P_k}[X \mid X\geq t] = \frac{\int_{t}^{\mu+k} \frac{z}{2k}{\rm d}z}{\int_{t}^{\mu+k} \frac{1}{2k}{\rm d}z} = \frac12(t+\mu+k) \stackrel{k\to\infty}{\longrightarrow} \infty,
$$
hence resulting in the third case. Combining the cases above completes the proof.
\end{proof}

\begin{proof}[Proof of Proposition~\ref{prop:prob}.]
The primal can be equivalently stated as
$$
\sup_{\mathbb{P}\in\mathcal{P}_{(\mu,\sigma)}} \E[\1_{\{X\geq z\}} \mid X\geq p],
$$
which is (weakly) dual to
\begin{equation}\label{eq:dual1p}
\begin{aligned}
&\inf_{\lambda_0,\lambda_1, \lambda_2, \lambda_3} &  & \lambda_3\\
&\text{subject to } &      & \lambda_0 + \lambda_1 \mu + \lambda_2(\sigma^2+\mu^2)\leq 0, \\
& &      & \lambda_0+\lambda_1 x+\lambda_2 x^2 \geq (\1_{\{x\geq z\}}(x) - \lambda_3 )\1_{\event}(x), \ \forall x\geq0.
\end{aligned}
\end{equation}
The right-hand side of the constraint is equal to 0, for $x<t$, $-\lambda_3$ for $t\leq x<z$, and $1-\lambda_3$ for $x\geq z$. We discuss three cases, in the order of their appearance in the claim. The first dual solution, $M_1(x)$, corresponds to a convex quadratic function that touches $(\1_{\{x\geq z\}}(x) - \lambda_3 )\1_{\event}(x)$ at $z$ and $x_0$, where the latter point lies between $p$ and $z$. The primal probabilities, dual variables, and the support point $x_0$ follow from solving
$$
\begin{aligned}
p_{x_0} + p_z = 1,\ p_{x_0} x_0 + p_z z = \mu,\ p_{x_0} x^2_0 + p_z z^2 = \mu^2 + \sigma^2, \\
\lambda_0 + \lambda_1 x_0 + \lambda_2 x_0^2 = -\lambda_3,\ \lambda_0 + \lambda_1 z + \lambda_2 z^2 = 1-\lambda_3,\\
\lambda_1 + 2\lambda_2 x_0 = 0,\ \lambda_0 + \lambda_1 \mu + \lambda_2 (\mu^2+\sigma^2)=0,
\end{aligned}
$$
yielding as solution
$$
\begin{aligned}
&p_{x_0} = \frac{(z-\mu )^2}{\sigma ^2+(z-\mu )^2},\ p_z = \frac{\sigma ^2}{\sigma ^2+(z-\mu )^2},\ x_0 = \mu +\frac{\sigma ^2}{\mu -z},\\
&\lambda_0 = \frac{(z-\mu ) \left(\mu ^2 (z-\mu )-\sigma ^2 (\mu +z)\right)}{\left(\sigma ^2+(z-\mu )^2\right)^2},\ \lambda_1 = \frac{2 (z-\mu ) \left(\mu ^2+\sigma ^2-\mu  z\right)}{\left(\sigma ^2+(z-\mu )^2\right)^2},\ \lambda_2 = \frac{(z-\mu )^2}{\left(\sigma ^2+(z-\mu )^2\right)^2},\ \lambda_3 = \frac{\sigma ^2}{\sigma ^2+(z-\mu )^2}.
\end{aligned}
$$
Indeed, by weak duality, this gives the best possible bound since
$$
\E[\1_{\{x\geq z\}}(X)] = p_{z} = \frac{\sigma ^2}{\sigma ^2+(z-\mu )^2} = \lambda_3.
$$
For the second case, let $M_2(x)$ denote a quadratic function that touches at $x=z$ and some point $x_0$, like $M_1(x)$, but additionally agrees with $(\1_{\{x\geq z\}}(x) - \lambda_3 )\1_{\event}(x)$ at $x=p$. Again, we can use a similar set of conditions, as described above, to find the optimal primal and dual solutions, but now with  a three-point distribution with $x_0=\frac{\mu ^2+\sigma ^2-\mu  z}{\mu -z}$ (which follows from the conditions). This leads to the second bound.
For brevity, we omit the detailed calculations.
Finally, for the third case, $M_3(x)$ is a quadratic function that touches at $z$ and a point $0\leq x\leq p$. The same set of calculations leads to the third upper bound, which is equal to the constant 1.
We combine the cases above in such a way that the primal distributions are feasible. This completes the proof.
\end{proof}

\section{Conic reformulations}\label{app:conics}
\begin{proof}[Proof of the second claim in Theorem~\ref{thm3}.]
We will focus on the second semi-infinite constraint (the first constraint can be dealt with analogously), which can equivalently be written as the collection of robust counterparts
\begin{equation}\label{eq:DeMorganConstraints4}
\lambda_0 + \vlambda_1^\top \vx + \vlambda_2^\top \vu  +\lambda_{3}  \geq \vec{s}_l(\gvec{\nu})^\top \vec{x}  + t_l(\gvec{\nu}), \quad \forall (\vec{x},\vu)\in\mathcal{C},\ \forall l\in\mathcal{L}.
\end{equation}
Let “$\operatorname{cl}$” denote the closure of a set. We next generate a proper cone from the uncertainty set as follows. Define ${\mathcal{K}}:=\operatorname{cl}\left(\left\{(\vec{x}, \vec{u}, w) \in \mathbb{R}^n \times \mathbb{R}^m \times \mathbb{R} : (\vec{z} / w, \vec{u} / w) \in {\mathcal{C}}, w>0\right\}\right)$, to which the cone ${\mathcal{K}}^*$ is dual. We henceforth assume that $\mathcal{K}$ and its dual cone ${\mathcal{K}}^*$ are representable as tractable cones.
The semi-infinite constraint \eqref{eq:DeMorganConstraints4} is satisfied if, and only if,
\begin{equation}\label{eq:subprimal}
\begin{aligned}
 & \inf  &    &\left(\vlambda_1-\vec{s}_l(\vnu)\right)^{\top} \vx +\vlambda_2^\top \vu \\
& \text { s.t. }  &    &(\vec{x}, \vec{u}, 1) \in {\mathcal{K}},
\end{aligned}
\end{equation}
is greater than, or equal to, $t_l(\gvec{\nu})-\lambda_0-\lambda_3$. Suppose there exists a strictly feasible solution to this problem. Then, by conic duality, the strong dual of \eqref{eq:subprimal} is given by
\begin{equation}\label{eq:subdual}
\begin{aligned}
& \sup &  & -w_l \\
& \text { s.t. } &  &  \vlambda_1-\vec{s}_l(\vnu)-\vec{a}_l=\vec{0}, \\
&&  &  \vlambda_2-\vec{b}_l=\vec{0}, \\
&&  &  \left(\vec{a}_l, \vec{b}_l, w_l\right) \in {\mathcal{K}}^*.
\end{aligned}
\end{equation}
Therefore, the semi-infinite constraint \eqref{eq:DeMorganConstraints4} is satisfied if, and only if, there exist solutions $\left(\vec{a}_l, \vec{b}_l, w_l\right) \in {\mathcal{K}}^*$, for all $l\in\mathcal{L}$, such that the constraints in \eqref{eq:subdual} are satisfied and $-w_l$ is not less than $t_l(\gvec{\nu})-\lambda_0-\lambda_3$. Using this dual characterization, we can rewrite \eqref{eq:contextualDRO} in the following way:
\begin{equation}\label{eq:contextualDROconicform}
\begin{aligned}
&\inf_{} &    & \lambda_{3}\\
&\textnormal{s.t. } &         &\lambda_0 + \vlambda_1^\top \gvec{\mu} + \vlambda_2^\top \gvec{\sigma} \leq 0, \\
&&&\lambda_0 + \vlambda_1^\top \vec{x} + \vlambda_2^\top \vu \geq 0, &\forall (\vec{x},\vu) \in \overline{\mathcal{C}},\\
  &&&\lambda_0 +\lambda_{3}  - t_l(\vx) -w_l  \geq 0, \quad &\forall l\in\mathcal{L}, \\
  &&&\vlambda_1-\vec{s}_l(\vnu)-\vec{a}_l=\vec{0}, \quad &\forall l\in\mathcal{L}, \\
  &&&\vlambda_2-\vec{b}_l=\vec{0}, \quad &\forall l\in\mathcal{L}, \\
  &&&\left(\vec{a}_l, \vec{b}_l, w_l\right) \in {\mathcal{K}}^*, \quad &\forall l\in\mathcal{L}, \\
  &&& \gvec{\nu}\in\mathcal{V},\ \lambda_0\in\mathbb{R},\  \gvec{\lambda}_1\in\mathbb{R}^n,\  \gvec{\lambda}_2\in\mathcal{D}^*,\ \lambda_3\in\mathbb{R}. &
\end{aligned}
\end{equation}
If $\overline{\mathcal{C}}$ is also conic representable and satisfies a similar Slater condition, an analogous argument enables us to reformulate the first semi-infinite constraint, reducing \eqref{eq:contextualDROconicform} to a finite-dimensional conic optimization problem. Hence, the second claim follows. 
\end{proof}

\medskip
We use the following result to derive the LMI reformulations for the Chebyshev ambiguity set $\cP_{(\gvec{\mu},\gvec{\Sigma})}$.
\begin{lemma}[S-Lemma, \cite{polik2007survey}] Consider two quadratic functions of $\vec{x} \in \mathbb{R}^n$, $q_i(\vec{x})=\vec{x}^{\top} \vec{C}_i \vec{x}+2 \vec{c}_i^{\top} \vec{x}+\bar{c}_i,\  i=0,1$, with $q_1(\overline{\vec{x}})>0$ for some $\overline{\vec{x}}$. Then
$$
q_0(\vec{x}) \geq 0 \quad \forall \vec{x}: q_1(\vec{x}) \geq 0
$$
if, and only if, there exists $\tau \geq 0$ such that
$$
\left(\begin{array}{cc}
\bar{c}_0 & \vec{c}_0^{\top} \\
\vec{c}_0 & \vec{C}_0
\end{array}\right)-\tau\left(\begin{array}{cc}
\bar{c}_1 & \vec{c}_1^{\top} \\
\vec{c}_1 & \vec{C}_1
\end{array}\right) \in \mathbb{S}_{+}^{n+1}.
$$
\end{lemma}
\medskip

\begin{proof}[Proof of Corollary~\ref{cor:Chebyshev}.]
By the S-Lemma, we have that
$$
\lambda_0 + \vlambda_1^{\top}\vx + \vx^{\top}\vec{\Lambda} \vx \geq 0,\ \forall \vx: \vec{c}^{\top}\vx \geq \bar{c}\iff \exists\tau\geq0:\ 
\begin{bmatrix}
\lambda_0 + \tau\bar{c} & \frac{1}{2}\left(\vlambda_1 - \tau \vec{c}\right)^{\top} \\
\frac{1}{2}\left(\vlambda_1 - \tau \vec{c}\right) & \boldsymbol{\Lambda}
\end{bmatrix}\succcurlyeq \mathbf{0}.
$$
Analogously, for the second semi-infinite constraint,
$$
\lambda_0 + (\vlambda_1-\vec{s}(\vnu))^{\top}\vx +\vx^{\top}\vec{\Lambda} \vx  +\lambda_3 - t_l(\gvec{\nu}) \geq 0,\ \forall \vx: \vec{c}^{\top}\vx \leq \bar{c}\iff\exists\chi_l\geq0:\ 
\begin{bmatrix}
\lambda_0 + \lambda_3 - t_l(\vnu) - \chi_l\bar{c} & \frac{1}{2}\left(\vlambda_1 - \vec{s}_l(\vnu) + \chi_l \vec{c}\right)^{\top} \\
\frac{1}{2}\left(\vlambda_1 - \vec{s}_l(\vnu) + \chi_l \vec{c}\right) & \boldsymbol{\Lambda}
\end{bmatrix}\succcurlyeq \mathbf{0}.
$$
These LMIs yield the  following semidefinite programming problem:
$$
\begin{aligned}
&\inf_{} &    & \lambda_{3}\\
&\textnormal{s.t. } &         &\lambda_0 + \vlambda_1^\top \gvec{\mu} + \langle\vec{\Lambda},\vec{\Sigma}\rangle \leq 0, \\
&&&\begin{bmatrix}
\lambda_0 + \tau\bar{c} & \frac{1}{2}\left(\vlambda_1 - \tau \vec{c}\right)^{\top} \\
\frac{1}{2}\left(\vlambda_1 - \tau \vec{c}\right) & \boldsymbol{\Lambda}
\end{bmatrix}\succcurlyeq \mathbf{0}, & \\
  &&&\begin{bmatrix}
\lambda_0 + \lambda_3 - t_l(\vnu) - \chi_l\bar{c} & \frac{1}{2}\left(\vlambda_1 - \vec{s}_l(\vnu) + \chi_l \vec{c}\right)^{\top} \\
\frac{1}{2}\left(\vlambda_1 - \vec{s}_l(\vnu) + \chi_l \vec{c}\right) & \boldsymbol{\Lambda}
\end{bmatrix}\succcurlyeq \mathbf{0}, \quad &\forall l\in\mathcal{L}, \\
  &&& \gvec{\nu}\in\mathcal{V},\ \lambda_0\in\mathbb{R},\  \gvec{\lambda}_1\in\mathbb{R}^n,\  \gvec{\Lambda}\in\mathbb{S}^{n}_{+},\ \lambda_3\in\mathbb{R},\ \tau\in\mathbb{R}_+,\ \gvec{\chi}\in\mathbb{R}^{|\mathcal{L}|}_+, &
\end{aligned}
$$
where $\langle\cdot,\cdot\rangle$ denotes the trace inner product. 
\end{proof}

\begin{proof}[Proof of Corollary~\ref{cor:MAD}.] The model with MAD information follows from defining separate constraints for the positive and negative parts of the absolute value terms. 
This yields
$$
\begin{aligned}
&\inf_{\gvec{\nu},\lambda_0,\vlambda_1,\vlambda_2,\lambda_3} &    & \lambda_{3}\\
&\textnormal{s.t. } &         &\lambda_0 + \vlambda_1^\top \vec{m} + \vlambda_2^\top \vec{f}\leq 0, \\
&&&\lambda_0 + \vlambda_1^\top \vec{x} + \vlambda_2^\top \vec{u} \geq 0, \quad &\forall (\vx,\vec{u}):\vec{c}^{\top}\vx\geq \bar{c},\ \vec{u}\geq \pm \vec{d}(\vx) \\
&&&\lambda_0 + \vlambda_1^\top \vec{x} + \vlambda_2^\top \vec{u}  +\lambda_{3}  \geq \vec{s}_l(\gvec{\nu})^\top \vec{x}  + t_l(\gvec{\nu}), \quad &\forall (\vx,\vec{u}):\vec{c}^{\top}\vx\leq\bar{c},\ \vec{u}\geq \pm \vec{d}(\vx)   \\
&&& \gvec{\nu}\in\mathcal{V},\ \lambda_0\in\mathbb{R},\  \gvec{\lambda}_1\in\mathbb{R}^n,\  \gvec{\lambda}_2\in\mathbb{R}^{n^2}_{+},\ \lambda_3\in\mathbb{R}, &
\end{aligned}
$$
where $\vd(\vx)=\vec{m}_0 + \vec{D}\vx$ describes the affine functions $X_i-m_i$ and $(X_i\pm X_k)-(m_i\pm m_k)$, and the inequalities $\vec{u}\geq \pm\vec{d}(\vec{x})$ hold elementwise. Let us focus on the first semi-infinite constraint. 
It can be rewritten as 
$$
\lambda_0 + \min_{\vx,\vec{u}:\vec{c}^{\top}\vx\geq \bar{c},\ \vec{u}\geq \pm \vec{d}(\vx)} \left\{ \vlambda_1^\top \vec{x} + \vlambda_2^\top \vec{u}\right\} \geq 0.
$$
Using the fact that $\vd$ is an affine, vector-valued function of $\vec{x}$, standard LP duality yields a finite-dimensional linear reformulation since 
$$
\begin{aligned}
&\lambda_0 + \min_{\vx,\vec{u}:\vec{c}^{\top}\vx\geq \bar{c},\ \vec{u}\geq \pm \vec{d}(\vx)} \left\{\vlambda_1^\top \vec{x} + \vlambda_2^\top \vec{u} \right\} \geq 0 \\
\iff& \lambda_0  + \max_{\tau\in\mathbb{R}_+, \boldsymbol{\chi}_+,\boldsymbol{\chi}_-\in\mathbb{R}_{+}^{n^2}}\left\{ \bar{c}\tau + \boldsymbol{\chi}_+^{\top}\vec{m}_0  - \boldsymbol{\chi}_-^{\top}\vec{m}_0 \ :\ \tau\vec{c} - \vec{D}^{\top} \boldsymbol{\chi}_+ + \vec{D}^{\top} \boldsymbol{\chi}_- = \vlambda_1,\  \boldsymbol{\chi}_+ + \boldsymbol{\chi}_- = \vlambda_2\right\}\geq 0\\
\iff& \exists \tau\in\mathbb{R}_+, \boldsymbol{\chi}_+,\boldsymbol{\chi}_-\in\mathbb{R}_{+}^{n^2} : \lambda_0 + \bar{c}\tau + \boldsymbol{\chi}_+^{\top}\vec{m}_0  - \boldsymbol{\chi}_-^{\top}\vec{m}_0 \geq0,\  \tau\vec{c} - \vec{D}^{\top} \boldsymbol{\chi}_+ + \vec{D}^{\top} \boldsymbol{\chi}_- = \vlambda_1,\  \boldsymbol{\chi}_+ + \boldsymbol{\chi}_- = \vlambda_2.
\end{aligned}
$$
A similar argument applies to the other semi-infinite constraint. Therefore, all robust counterparts can be rewritten in terms of linear inequalities, yielding the result.
\end{proof}

\end{appendix}

\bibliography{arxiv.bib}
\bibliographystyle{apalike}

\end{document}